\documentclass[12pt,american,english]{amsart}
\usepackage[T1]{fontenc}
\usepackage[latin9]{inputenc}
\usepackage{geometry}
\usepackage{babel}
\usepackage{refstyle}
\usepackage{amsthm}
\usepackage{amssymb}
\usepackage{nomencl}

\providecommand{\makenomenclature}{\makeglossary}
\makenomenclature
\usepackage[unicode=true,pdfusetitle,
 bookmarks=true,bookmarksnumbered=false,bookmarksopen=false,
 breaklinks=false,pdfborder={0 0 1},backref=false,colorlinks=false]
 {hyperref}

\makeatletter

\AtBeginDocument{\providecommand\thmref[1]{\ref{thm:#1}}}
\AtBeginDocument{\providecommand\secref[1]{\ref{sec:#1}}}
\AtBeginDocument{\providecommand\lemref[1]{\ref{lem:#1}}}
\AtBeginDocument{\providecommand\propref[1]{\ref{prop:#1}}}
\AtBeginDocument{\providecommand\subref[1]{\ref{sub:#1}}}
\AtBeginDocument{\providecommand\corref[1]{\ref{cor:#1}}}
\RS@ifundefined{subref}
  {\def\RSsubtxt{section~}\newref{sub}{name = \RSsubtxt}}
  {}
\RS@ifundefined{thmref}
  {\def\RSthmtxt{theorem~}\newref{thm}{name = \RSthmtxt}}
  {}
\RS@ifundefined{lemref}
  {\def\RSlemtxt{lemma~}\newref{lem}{name = \RSlemtxt}}
  {}

\numberwithin{equation}{section}
\numberwithin{figure}{section}
\theoremstyle{plain}
\newtheorem{thm}{\protect\theoremname}[section]
  \theoremstyle{definition}
  \newtheorem{defn}[thm]{\protect\definitionname}
  \theoremstyle{remark}
  \newtheorem{rem}[thm]{\protect\remarkname}
  \theoremstyle{plain}
  \newtheorem{prop}[thm]{\protect\propositionname}
  \theoremstyle{plain}
  \newtheorem{lem}[thm]{\protect\lemmaname}
  \theoremstyle{remark}
  \newtheorem*{rem*}{\protect\remarkname}
  \theoremstyle{plain}
  \newtheorem{cor}[thm]{\protect\corollaryname}

\usepackage{ytableau}
\usepackage{xypic}

\makeatother

  \addto\captionsamerican{\renewcommand{\corollaryname}{Corollary}}
  \addto\captionsamerican{\renewcommand{\definitionname}{Definition}}
  \addto\captionsamerican{\renewcommand{\lemmaname}{Lemma}}
  \addto\captionsamerican{\renewcommand{\propositionname}{Proposition}}
  \addto\captionsamerican{\renewcommand{\remarkname}{Remark}}
  \addto\captionsamerican{\renewcommand{\theoremname}{Theorem}}
  \addto\captionsenglish{\renewcommand{\corollaryname}{Corollary}}
  \addto\captionsenglish{\renewcommand{\definitionname}{Definition}}
  \addto\captionsenglish{\renewcommand{\lemmaname}{Lemma}}
  \addto\captionsenglish{\renewcommand{\propositionname}{Proposition}}
  \addto\captionsenglish{\renewcommand{\remarkname}{Remark}}
  \addto\captionsenglish{\renewcommand{\theoremname}{Theorem}}
  \providecommand{\corollaryname}{Corollary}
  \providecommand{\definitionname}{Definition}
  \providecommand{\lemmaname}{Lemma}
  \providecommand{\propositionname}{Proposition}
  \providecommand{\remarkname}{Remark}
\providecommand{\theoremname}{Theorem}

\begin{document}

\title{Kemer's Theory for $H$-Module Algebras with Application to the PI
Exponent}

\author{Yaakov Karasik}

\address{Department of Mathematics, Technion - Israel Institute of Technology,
Haifa 32000, Israel}

\email{yaakov@tx.technion.ac.il}

\begin{abstract}
Let $H$ be a semisimple finite dimensional Hopf algebra over a field
$F$ of zero characteristic. We prove three major theorems:. 1. The
Representability theorem which states that every $H$-module (associative)
$F$-algebra $W$ satisfying an ordinary PI, has the same $H$-identities
as the Grassmann envelope of an $H\otimes\left(F\mathbb{Z}/2\mathbb{Z}\right)^{*}$-module
algebra which is finite dimensional over a field extension of $F$.
2. The Specht problem for $H$-module (ordinary) PI algebras. That
is, every $H$-$T$-ideal $\Gamma$ which contains an ordinary PI
contains $H$-polynomials $f_{1},...,f_{s}$ which generates $\Gamma$
as an $H$-$T$-ideal. 3. Amitsur's conjecture for $H$-module algebras,
saying that the exponent of the $H$-codimension sequence of an ordinary
PI $H$-module algebra is an integer.
\end{abstract}

\maketitle

\section{Introduction \label{sec:Introduction}}

Two of the main problems in the theory of asssociative algebras satisfying
a polynomial identity (PI in short) are the Specht problem (see \cite{Specht1950})
and the Representability theorem (\cite{Kemer1984}). The classical
Specht problem asks whether a $T$-ideal can be generated as a $T$-ideal
by a finite number of polynomials. The Representability theorem states
that every PI algebra has the same identities (PI equivalent) as the
Grassmann envelope of a $\mathbb{Z}/2\mathbb{Z}$-graded finite dimensional
algebra. Moreover, if the given PI algebra is affine, then it is PI
equivalent to a finite dimensional algebra. The two theorems seems
unrelated, since there is no obvious reason for a $T$-ideal of identities
of (even) a finite dimensional algebra to be finitely based. However,
both of them were solved in the 80's by Kemer \cite{Kemer1984} using
the same ideas. Thus intertwining the two problems.

In the recent decades different classes of algebras, such as non-associative
algebras, group graded algebras, group acted algebras, algebras with
involution, were studied in the context of PI theory. In all of these
frameworks analogs of these problems exist and in some of them also
solved: For finite group-graded algebras satisfying an ordinary PI
see \cite{Aljadeff2010a} (it is worth mentioning that in \cite{Sviridova2011}
the special case of abilean finite groups is treated). For algebras
with involutions satisfying an ordinary PI see \cite{Sviridova2013}.
For affine algebras over fields of non-zero characteristic see \cite{AlexeiBelov-Kanel}.
The assumption that the algebra satisfies an ordinary PI (and not
just a PI of the framework in consideration) is essential for the
Representability theorem, since finite dimensional algebras and the
Grassmann envelope of a finite dimensional algebras are satisfying
an ordinary PI. However, it might be the case that the Specht problem
remains true without this assumption. 

In this paper we work in the framework of $H$-module algebras satisfying
an ordinary PI, where $H$ is a finite dimensional and semisimple
Hopf $F$-algebra ($F$ is a characteristic zero field). Two important
examples of families of algebras which this framework generalizes
are the (finite) group graded algebras and the group acted algebras:
Suppose $G$ is any finite group. By considering $H$ to be the dual
of the group algebra $FG$ we obtain the family of $G$-graded algebras;
whereas by considering $H=FG$ we obtain the family of algebras with
a $G$ action (by $F$-algebra automorphisms). So far the Specht and
Representability problems were open for the latter family in the case
where $G$ is non-abelian. If $G$ is abelian, then these problems
are equivalent to the corresponding problems in the $G$-graded case
(same $G$). 

Let us introduce the notation to discuss these problems. Suppose $H$
is an $m$-dimensional Hopf algebra over a field $F$ of characteristic
zero and let $W$ be an $H$-module algebra over $F$. Suppose $X=\{x_{1},...,x_{n},...\}$
is a set of non-commutative variables and consider the vector space
$V=FX\otimes_{F}H$. An $H$-polynomial is an element in the tensor
algebra (without $1$) over $V$, which we denote by $F^{H}\left\langle X\right\rangle $.
One might prefer instead a coordinate oriented definition of $F^{H}\left\langle X\right\rangle $:
Choose a basis $\{b_{1},...b_{m}\}$ for the $F$-algebra $H$. Then
$F^{H}\left\langle X\right\rangle $ is understood as the $F$-algebra
generated by the formal (non-commutative) variables $x^{b_{i}}$,
where $i\in\{1,...,m\}$ and $x\in X$. Notice that $F^{H}\left\langle X\right\rangle $
is an $H$-module algebra, where 
\[
h\cdot\left((x_{i_{1}}\otimes h_{1})\otimes\cdots\otimes(x_{i_{k}}\otimes h_{k})\right)=(x_{i_{1}}\otimes h_{(1)}h_{1})\otimes\cdots\otimes(x_{i_{k}}\otimes h_{(k)}h_{k})
\]
or 
\[
h\cdot x_{i_{1}}^{h_{1}}\cdots x_{i_{k}}^{h_{k}}=x_{i_{1}}^{h_{(1)}h_{1}}\cdots x_{i_{k}}^{h_{(k)}h_{k}},
\]
where $h_{1},...,h_{k}\in H$ (we use the Swidler notation: $\Delta(h)=h_{(1)}\otimes h_{(2)}$). 

We say that $f\in F^{H}\left\langle X\right\rangle $ is an identity
of $W$ if for every $H$-homomorphism $\phi:F^{H}\left\langle X\right\rangle \rightarrow W$
the polynomial $f$ is in the kernel of $\phi$. Put differently,
$f$ is an identity of $W$ if $f$ vanishes for every substitution
of the variables from $X$ by elements of $W$. The set of all identities,
denoted by $id^{H}(W)$, is an ideal of $F^{H}\left\{ X\right\} $
which is also stable under $H$-endomorphisms. Such an ideal is called
$H$-$T$-ideal. 

Finally, suppose $W_{1}$ and $W_{2}$ are two $H$-module $F$-algebras.
We say that $W_{1}\sim_{H-PI}W_{2}$ ($H$-PI equivalent) if $id^{H}(W_{1})=id^{H}(W_{2})$.
It is crucial to notice that $W\sim_{H-PI}\mathcal{W}$, where $\mathcal{W}$
(always) denotes the relatively free $H$-module algebra $F^{H}\left\{ X\right\} /id^{H}(W)$. 

The main part of this paper is dedicated to proving the following
theorem:
\begin{thm}[Affine $H$-Representability]
\label{thm:affine-rep-1} Let $W$ be an affine $H$-module algebra
over a field $F$ of characteristic zero satisfying an ordinary polynomial
identity, where $H$ is a finite dimensional semisimple Hopf $F$-algebra.
Then there exists a field extension $L$ of $F$ and a finite dimensional
$H$-module algebra $A$ over $L$ which is $H$-PI equivalent to
$W$. 
\end{thm}
To state the general $H$-representability theorem we need more notations.
Denote by $E=E_{0}\oplus E_{1}$ the Grassmann superalgebra over $F$.
Suppose $W$ is an $H_{2}=H\otimes_{F}(F\mathbb{Z}/2\mathbb{Z})^{*}$-module
algebra. In other words, $W=W_{0}\oplus W_{1}$ is a superalgebra
endued with $H$-module algebras structure such that $W_{0}$ and
$W_{1}$ are stable under the action of $H$. The Grassmann envelope
of $W$ is the $H_{2}$-module $F$-algebra $E(W)=(W_{0}\otimes E_{0})\oplus(W_{1}\otimes E_{1})$.
The $H$-representability theorem states:
\begin{thm}[H-Representability]
\label{thm:Hrep!!} Let $W$ be an $H$-module algebra over a field
$F$ of characteristic zero satisfying an ordinary polynomial identity,
where $H$ is a finite dimensional semisimple Hopf $F$-algebra. Then
there exists a field extension $L$ of $F$ and a finite dimensional
$H_{2}$-module algebra $A$ over $L$ such that $W\sim_{H-PI}E(A)$. 
\end{thm}
In the final section of this paper we obtain:
\begin{thm}[Specht]
\label{thm:Specht!!!} Suppose $\Gamma$ is an $H$-$T$-ideal containing
an ordinary identity, then there are $f_{1},...,f_{s}\in\Gamma$ which
$H$-$T$-generate $\Gamma$. Equivalently, if $\Gamma_{1}\subseteq\Gamma_{2}\subseteq\cdots$
is an ascending chain of $H$-$T$-ideals containing an ordinary PI,
then the chain stabilizes. 
\end{thm}
The first and main part of this article is the proof of \thmref{affine-rep-1}.
For this we follow, for the most part, the exposition of Kemer's proof
given in \cite{Aljadeff}. However, there are two major differences.
The first is the proof of ``Kemer Lemma 1'' and the second is the
construction of ``representable spaces'' for the Kemer polynomials
(see \secref{Sketch-of-the} for details). The conclusion of \thmref{Hrep!!}
and \thmref{Specht!!!} is completely standard and we use the same
argument as in \cite{Aljadeff2010a,Kanel-Belov2005}.

Let us recall the definition of the $H$-codimension sequence of an
$H$-module algebra:
\begin{defn}
Let $W$ be an $H$-module $F$-algebra. The $H$-codimension of $W$
is 
\[
c_{n}^{H}(W)=\dim_{F}P_{n}^{H}/P_{n}^{H}\cap id^{H}(W),
\]
where $P_{n}^{H}$ is the $F$-space spanned by $x_{\sigma(1)}^{h_{1}}\cdots x_{\sigma(n)}^{h_{n}}$,
where $\sigma\in S_{n}$ and $h_{1},...,h_{n}\in H$. 
\end{defn}
A consequence of \thmref{Hrep!!} is the affirmative solution of Amitsur's
conjecture on the exponent in the case of general $H$-module $F$-algebras. 
\begin{thm}[$H$-Amitsur's Conjecture]
\label{thm:exp} Suppose $W$ is any $H$-module $F$-algebra which
satisfies an ordinart PI, then the $H$-exponent of $W$ defined by
\[
\exp^{H}(W)=\underset{n\to\infty}{\lim}\sqrt[n]{c_{n}^{H}(W)}
\]
exist and is an integer. 
\end{thm}
Using the ideas of Gordianko and Zaicev in \cite{Giambruno2005} and
Gordienko in \cite{Gordienko2013} we will obtain this theorem in
the final section.

\selectlanguage{american}%

\section{Preliminaries}

\selectlanguage{english}%
There are two families of $H$-polynomails which play a leading role
in PI theory: multilinear and alternating polynomials:
\begin{defn}
$f(x_{1},...,x_{n})\in F^{H}\left\{ X\right\} $ is \emph{multilinear}
if 
\[
f(x_{1},...,x_{i-1},\alpha x_{i}+y,x_{i+1},...,x_{n})=\alpha f(x_{1},...,x_{n})+f(x_{1},...,x_{i-1},y,x_{i},...,x_{n})
\]
for every $i$ between $1$ to $n$ and $\alpha\in F$.\end{defn}
\begin{rem}
In the case where $H$ is the dual of the group algebra $FG$ (here
$G$ is a finite group) the polynomial:
\[
f(x,y)=x_{g}y+y_{g}x_{e}
\]
is multilinear by our definition.\end{rem}
\begin{defn}
Let $f=f(x_{1},...,x_{n})\in F^{H}\left\{ X\right\} $. For $1\leq i,j\leq n$
we denote by $f|_{x_{i}\to x_{j}}=f|_{x_{i}=x_{j}}$ the polynomial
obtained from $f$ by substituting $x_{i}$ inside $x_{j}$. Moreover,
$f|_{x_{i}\leftrightarrow x_{j}}$ denotes the polynomial obtained
from $f$ by replacing $x_{i}$ by $x_{j}$and vice versa. 
\end{defn}

\begin{defn}
Suppose $f=f(x_{1},...,x_{n},Y)\in F^{H}\left\{ X\right\} $, where
$Y$ is a set of variables disjoint from $x_{1},...x_{n}$. We say
that $f$ is alternating on $x_{1},...,x_{n}$ if
\[
f|_{x_{i}\leftrightarrow x_{j}}=-f
\]
for every $i$ and $j$ between $1$ to $n$. Since the characteristic
of $F$ is not $2$ this is equivalent to 
\[
f|_{x_{i}=x_{j}}=0.
\]

If $f=f(X,Y)\in F^{H}\left\{ X\right\} $ is any polynomial we define
\[
\mbox{Alt}_{X}(f)=\sum_{\sigma\in S_{X}}(-1)^{\sigma}f|_{x\in X\leftarrow\sigma(x)}.
\]
\nomenclature[97]{$\mbox{Alt}_{X}(f)$}{The polynomial obtained from $f$ by alternating the variables from the set $X$}Therefore,
$\mbox{Alt}_{X}(f)$ is alternating on $X$. If $f$ was alternating
on $X$ to begin with, then $\mbox{Alt}_{X}(f)=|X|!\cdot f$.
\begin{rem}
As in classical PI theory any $H$-$T$-ideal $\Gamma$ is $T$-generated
by the multilinear polynomials inside $\Gamma$.
\end{rem}
\end{defn}
Finally, suppose $W_{1}$ and $W_{2}$ are two $H$-module $F$-algebras.
We say that $W_{1}\sim_{H-PI}W_{2}$ ($H$-PI equivalent) if $id^{H}(W_{1})=id^{H}(W_{2})$.
It is crucial to notice that $W\sim_{H-PI}\mathcal{W}$, where $\mathcal{W}$
(always) denotes the relatively free $H$-module algebra $F^{H}\left\{ X\right\} /id^{H}(W)$.

\section{Sketch of the proof of \thmref{affine-rep-1} \label{sec:Sketch-of-the}}

In this short section we outline the main steps of the proof of \thmref{affine-rep-1}.
\begin{enumerate}
\item Every affine (ordinary) PI $H$-module $F$-algebra $W$ has a a finite
dimensional $H$-module $F$-algebra $A$ such that $id^{H}(A)\subseteq id^{H}(W)$. 
\item Definition of the $H$-Kemer index $\mbox{Ind}(\Gamma)=(\alpha,r)\in\Omega=\mathbb{Z}^{\ge0}\times\mathbb{Z}^{\ge0}$
\nomenclature[98]{$\mbox{Ind}(\Gamma)=\left(\alpha(\Gamma),s(\Gamma)\right)$}{The $H$-Kemer index of the $H-T$-ideal $\Gamma$}and
$H$-Kemer polynomials for $H$ $T$-ideals of $H$-module algebras
satisfying some Capelli identity. Since by the previous step any affine
PI $H$-module algebra satisfies a Capelli identity, the index is
defined for all the algebras under consideration.

Considering the lexicographic ordering ($\leq$) on $\Omega$ it will
be easy to conclude that if $\Gamma_{1}\subseteq\Gamma_{2}$ then
$\mbox{Ind}(\Gamma_{1})\leq\mbox{Ind}(\Gamma_{2})$ (reverse ordering). 

\item Construction of $H$-basic algebras. Every $H$-basic $H$-module
algebra $A$ is finite dimensional and has the property $\mbox{Ind}(A)=\mbox{Par}(A)=(d,s-1)$,
where $d$ is the dimension of the semisimple part of $A$ and $s$
is the nilpotency of $J(A)$, the radical of $A$. We show that every
finite dimensional $H$-module algebra is $H$-PI equivalent to a
finite direct product of $H$-basic algebras. As far as the author
knows, this step in all other frameworks (e.g. group graded algebras,
algebras with involutions) relies heavily on precise knowledge of
all the simple, finite dimensional objects of the category in question
(see \cite{Aljadeff,Aljadeff2010a}). However, in such general framework
as $H$-module algebras it seems that one must consider more ``subtle''
approach. Luckily, such approach was already introduced for different
purpose by Gordienko in \cite{Gordienko2013}. 
\item There is a finite dimensional $H$-module algebra $B$ having the
same $H$-Kemer index and $H$-Kemer polynomials as $W$. 
\item Using steps $3$ and $4$ the Phoenix property for $H$-$T$-ideals
will follow. This property states that if $f\notin\Gamma$ is a consequence
of an $H$-Kemer polynomial of $\Gamma$, then although $f$ might
fail being an $H$-Kemer polynomial, yet it has a consequence $f^{\prime}$
which is an $H$-Kemer polynomial of $\Gamma$. 
\item Construction of a \textbf{representable} $H$-module algebra $B_{\Gamma}$
satisfying the properties:

\begin{itemize}
\item $id^{H}(B_{\Gamma})\supseteq\Gamma$.
\item All $H$-Kemer polynomials of $\Gamma$ are non-identities of $B_{\Gamma}$.
\end{itemize}
\item We finalize the proof. consider $\Gamma^{\prime}=\Gamma+S$, where
$S$ is the $H$-$T$-ideal \textit{generated} by all $H$-Kemer polynomials
of $\Gamma$. This will imply that $\mbox{Ind}(\Gamma^{\prime})<\mbox{Ind}(\Gamma)$
and hence by induction on the $H$-Kemer index there exists a finite
dimensional $H$-module algebra $A'$ with $\Gamma'=id^{H}(A')$.
We show that \textit{all} polynomials of $S$ (which are not in $\Gamma$)
are nonidentities of $B_{\Gamma}$ (that is, not just elements in
$S$ which are $H$-Kemer polynomials). This is achieved by the Phoenix
property of Kemer polynomials. Since any nonidentity $f'$ of $\Gamma$
which is in $S$, produces (by the $T$-operation) a Kemer polynomial
which by Step $5$ is not in $id^{H}(B_{\Gamma})$ we have also that
$f'\notin id^{H}(B_{\Gamma})$. From that one concludes that $\Gamma=id^{H}(A'+B_{\Gamma})$.
\end{enumerate}

\section{Getting started\label{sec:Getting-started}\label{sec:the_index_of_=000024T=000024-ideals}}
\begin{thm}
Suppose $W$ is an affine $H$-module algebra which satisfies an ordinary
PI, then there is a finite dimensional $H$-module $F$-algebra $A$
such that $id^{H}(A)\subseteq id^{H}(W)$.\end{thm}
\begin{proof}
By the classical PI theory (see Corollary 4.9 in \cite{Kanel-Belov2005})
there is an $F$-algebra $A_{0}$ with the property $id(A_{0})\subseteq id(W)$.
Consider the $H$-module algebra $A=A_{0}\otimes H^{*}$, where the
$H$-action is given by 
\[
h(a\otimes\phi)=a\otimes\phi_{h},\,\phi_{h}(g)=\phi(gh)\,\,\, g,h\in H,\phi\in H^{*}.
\]
Since 
\begin{eqnarray*}
\left(h\left(\phi\cdot\psi\right)\right)(g) & = & \left(\phi\cdot\psi\right)(gh)=\phi(g_{(1)}h_{(1)})\psi(g_{(2)}h_{(2)})=\phi_{h_{(1)}}(g_{(1)})\phi_{h_{(2)}}(g_{(2)})\\
 & = & \left(\phi_{h_{(1)}}\cdot\phi_{h_{(2)}}\right)(g)=\left(h_{(1)}\phi\cdot h_{(2)}\psi\right)(g)
\end{eqnarray*}
and 
\[
\left(h\cdot\epsilon\right)(g)=\epsilon(gh)=\epsilon(h)\epsilon(g),
\]
we indeed defined an $H$-action. 

Suppose $f=f(x_{1},...,x_{n})\in id^{H}(A)$. We need to show that
$f\in id^{H}(W)$. Let $\phi\in H^{*}$ be defined by $\phi(h_{i})=\delta_{1,i}\in F$,
where $i=1,...,m=\dim_{F}H$. The important property of $\phi$ is
that $h_{1}\cdot\phi,...,h_{m}\cdot\phi$ are linearly independent
over $F$. Consider the substitution $\bar{x}_{1}=a_{1}\otimes\phi,\bar{x}_{2}=a_{2}\otimes1,...,\bar{x}_{n}=a_{n}\otimes1$
(here $1$ is the functional of $H$ which equals to $1$ at every
point). We obtain 
\[
f(\bar{x}_{1},...,\bar{x}_{n})=\sum_{i=1}^{m}g_{i}(a_{1},...,a_{n})\otimes h_{i}\cdot\phi,
\]
where $g_{i}\in F^{H}\left\{ X\right\} $ is multilinear polynomial
all of whose monomials contain the variable $x_{1}^{h_{i}}$. Therefore,
$g_{i}\in id^{H}(A)$. We may replace $x_{1}^{h_{i}}$ by $x_{1}$
and obtain $g_{i}^{(1)}\in id^{H}(A)$. Notice that it suffices to
show that $g_{1}^{(1)},...,g_{m}^{(1)}\in id^{H}(W)$. 

Repeat the argument for each one of the polynomials $g_{1}^{(1)},...,g_{m}^{(1)}$,
by considering the substitution $\bar{x}_{1}=a_{1}\otimes1,\bar{x}_{2}=a_{2}\otimes\phi,\bar{x}_{3}=a_{3}\otimes1,...,\bar{x}_{n}=a_{n}\otimes1$.
This will result in multilinear polynomials $g_{1}^{(2)},...,g_{m^{2}}^{(2)}\in id^{H}(A)$
having the properties: 
\begin{itemize}
\item All the monomials of each $g_{i}^{(2)}$ contain $x_{1}$ and $x_{2}$.
\item If $g_{1}^{(2)},...,g_{m^{2}}^{(2)}\in id^{H}(A)$, then $f\in id^{H}(A)$.
\end{itemize}

\selectlanguage{american}%
Repeating this argument \foreignlanguage{english}{eventually results
in the conclusion that $f\in id^{H}(W)$ if and only if some (ordinary!)
polynomials $g_{1}^{(n)},...,g_{m^{n}}^{(n)}\in id(A)$ are in $id(W)$.
However, this indeed holds due to the assumption on $A$.}

\end{proof}
\selectlanguage{english}%
\begin{defn}
\selectlanguage{english}%
Let $W$ be an $H$-module $F$-algebra. We say that $W$ satisfies
a Capelli identity $m$ if every $H$-polynomial $f(x_{1},...,x_{m},Y)$
which is alternating in $x_{1},...,x_{m}$ is in $id^{H}(W)$.
\end{defn}
The following definition of $H$-Kemer index and $H$-Kemer polynomials
makes sense only for $H$-module algebras satisfying a Capelli identity.
As we saw previously, this includes the affine $H$-module algebras
which satify an ordinary PI. 
\begin{defn}
Suppose $\Gamma$ satisfies some Capelli identity. Define $\alpha(\Gamma)$
to be the maximal integer such that for every $\mu$ there is a multilinear
polynomial $f=f(X_{1},...,X_{\mu},Y)\notin\Gamma$ which is alternating
with respect to the sets $X_{1},...X_{\mu}$ which are all of cardinality
$\alpha(\Gamma)$. 

$s(\Gamma)$ is defined as the maximal integer such that for every
$\nu$ there is a multilinear $g=g(X_{1},...,X_{\mu},X_{1}^{\prime},...,X_{s(\Gamma)}^{\prime},Y)\notin\Gamma$
which is alternating with respect to $X_{1},...,X_{\mu},X_{1}^{\prime},...,X_{s(\Gamma)}^{\prime}$,
where $|X_{1}|=\cdots=|X_{\mu}|=\alpha(\Gamma)$ and $|X_{1}^{\prime}|=\cdots=|X_{s(\Gamma)}^{\prime}|=\alpha(\Gamma)+1$.

We call the pair $\left(\alpha(\Gamma),s(\Gamma))\right)$ the $H$-Kemer
\emph{index }of $\Gamma$ and denote it by $\mbox{Ind}(\Gamma)$.
Any such $g$ is called $H$-Kemer polynomial of $\Gamma$ of \emph{rank}
$\mu$. We refer to $X_{1},...,X_{\mu}$ as \emph{small sets} and
to $X_{1}^{\prime},...,X_{s(\Gamma)}^{\prime}$ as \emph{big sets}.\end{defn}
\begin{rem}
\label{rem:comparison-1}If $\Gamma_{1}\subseteq\Gamma_{2}$ then
$\mbox{Ind}(\Gamma_{1})\geq\mbox{Ind}(\Gamma_{2})$ i.e. the order
is reversed. 
\end{rem}

\begin{rem}
In what follows we will always assume that $\mu\geq\mu_{\Gamma}$
where $\mu_{\Gamma}$ is the minimal integer for which any multilinear
$f=f(X_{1},...,X_{\mu_{\Gamma}},X_{1}^{\prime},...,X_{s(\Gamma)+1}^{\prime},Y)\in F^{H}\left\{ X\right\} $,
which alternates on $X_{1},...,X_{\mu_{\Gamma}},X_{1}^{\prime},...,X_{s(\Gamma)}^{\prime}$,
and $|X_{1}|=\cdots=|X_{\mu_{\Gamma}}|=\alpha(\Gamma)$, $|X_{1}^{\prime}|=\cdots=|X_{s(\Gamma)+1}^{\prime}|=\alpha(\Gamma)+1$,
is an identity of $\Gamma$.
\end{rem}

\section{The index of finite dimensional algebras\label{sec:The_index_of_finite_dimensional_algebra}}

We start this section with the definition of the Phoenix property.
\begin{defn}
(The Phoenix property) Let $\Gamma$ be an $H$-$T$-ideal as above.
Let $P$ be any property which may be satisfied by polynomials (e.g.
being $H$-Kemer). We say that $P$ is ``\textit{$\Gamma$-Phoenix}''
(or in short ``\textit{Phoenix}'') if given a multilinear polynomial
$f$ having $P$ which is not in $\Gamma$ and \textit{any} $f^{'}$
in $\langle f\rangle_{H}$\nomenclature[99]{$\langle f\rangle_{H}$ }{The $H-T$-ideal generated by the polynomial $f$}
(the $H$-$T$-ideal generated by $f$) which is not in $\Gamma$
as well, there exists a multilinear polynomial $f^{''}$ in $\langle f^{'}\rangle_{H}$
which is not in $\Gamma$ and satisfies $P$. We say that $P$ is
``\textit{strictly $\Gamma$-Phoenix}'' if any multilinear polynomial
$f^{'}\in\langle f\rangle_{H}$ which is not in $\Gamma$, satisfies
$P$.\end{defn}
\begin{rem}
Given a polynomial $g$, there exists a multilinear polynomial $f'$
such that $\langle f'\rangle_{H}=\langle g\rangle_{H}$. It follows
that in order to verify the Phoenix property it is sufficient to consider
multilinear polynomials $f'$ in $\langle f\rangle_{H}$.
\end{rem}
Let us pause for a moment and summarize what we have at this point.
We are given an $H$-$T$-ideal $\Gamma$ (the $T$-ideal of identities
of an affine $H$-module algebra $W$). We assume that $W$ is ordinery
\textit{PI} and hence as shown in \secref{Getting-started} there
exists a finite dimensional $H$-module algebra $A$ with $\Gamma\supseteq id^{H}(A)$.
To the $H$-$T$-ideal $\Gamma$ we attach the corresponding $H$-Kemer
index in $\mathbb{Z}^{\geq0}\times\mathbb{Z}^{\geq0}$. Similarly,
we may consider the Kemer index of $id^{H}(A)$ which by abuse of
notation we denote it by $\mbox{Ind}(A)$. Clearly, we have $\mbox{Ind}(\Gamma)\leq\mbox{Ind}(A)$.

One of our main goals (in the first part of the proof) is to replace
the $H$-module algebra $A$ by an $H$-module algebra $A^{'}$ with
a larger $T$-ideal such that
\begin{enumerate}
\item $\Gamma\supseteq id^{H}\left(A^{'}\right)$
\item $\Gamma$ and $iid^{H}\left(A^{'}\right)$ have the same $H$-Kemer
index.
\item $\Gamma$ and $id^{H}\left(A^{'}\right)$ have the ``same'' $H$-Kemer
polynomials.\end{enumerate}
\begin{rem}
The terminology ``the same $H$-Kemer polynomials'' needs a clarification.
If $\Gamma_{1}\supseteq\Gamma_{2}$ are $H-T$ ideals with $\mbox{Ind}(\Gamma_{1})=\mbox{Ind}(\Gamma_{2})$.
We say that $\Gamma_{1}$ and $\Gamma_{2}$ have the same $H$-Kemer
polynomials if there exists an integer $\mu$ such that all Kemer
polynomials of $\Gamma_{2}$ with at least $\mu$ alternating small
sets are not in $\Gamma_{1}$. Write $\mu_{\Gamma,\Gamma^{\prime}}$
for the maximum between the above $\mu$, $\mu_{\Gamma}$ and $\mu_{\Gamma^{\prime}}$.
\begin{rem}
Statements $(1)-(3)$ above will establish the important connection
between the combinatorics of the $H$-Kemer polynomials of $\Gamma$
and the structure of finite dimensional $H$-module algebras. The
``Phoenix'' property for the $H$-Kemer polynomials of $\Gamma$
will follow from that connection. 
\end{rem}
\end{rem}
Let $A$ be a finite dimensional $H$-module algebra over $F$ and
let $J(A)$ be its Jacobson radical. We know (\cite{Linchenko2005})
that $J(A)$\nomenclature[991]{$J(A)$}{The Jacobson radical of $A$}
is $H$-invariant, thus $\overline{A}=A/J(A)$\nomenclature[992]{$\overline{A}$}{The semisimple part of $A$}
is a semisimple $H$-module algebra. Moreover by the $H$-invariant
Wedderburn-Malcev Principal Theorem (see \cite{cStefan1999}) there
exists a semisimple $H$-module subalgebra $\overline{A}$ of $A$
such that $A=\overline{A}\oplus J(A)$ as vector spaces. In addition,
the subalgebra $\overline{A}$ may be decomposed as an algebra into
the direct product of $H$-simple algebras $\overline{A}\cong A_{1}\times A_{2}\times\cdots\times A_{q}$
(see {[}\cite{Giambruno2005}, Lemma 3{]}).
\begin{rem}
This decomposition enables us to consider ``semisimple'' and ``radical''
substitutions. More precisely, since in order to check whether a given
multilinear $H$-polynomial is an identity of $A$ it is sufficient
to evaluate the variables on any (given) spanning set, we may take
a basis consisting of elements of $\overline{A}\cup J(A)$. We refer
to such evaluations as semisimple or radical evaluations respectively.
Moreover, the semisimple substitutions may be taken from the simple
components.

\textit{In what follows, whenever we evaluate a polynomial on a finite
dimensional $H$-module algebra, we consider only evaluations of that
kind.} \bigskip{}

\end{rem}
For any finite dimensional $H$-module algebra $A$ over $F$ we let
$d(A)$ be the dimension of the semisimple subalgebra and $n_{A}$
the nilpotency index of $J(A)$. We denote by $\mbox{Par}(A)=(d(A),n_{A}-1)$
the \emph{parameter} of the $H$-module algebra $A$.
\begin{prop}
\label{prop:index-par_inequality}Let $(\alpha,s)$ be the index of
$A$. Then $(\alpha,s)\leq(d(A),n_{A}-1)$.\end{prop}
\begin{proof}
By the definition of the parameter $\alpha$, there exist nonidentity
polynomials with arbitrary large number of alternating sets of cardinality
$\alpha$. Now, if $\alpha>d(A)$ any such alternating set must have
at least one radical evaluation and hence the polynomial cannot have
more than $(n_{A}-1)$ alternating sets of cardinality $\alpha$.
Contradiction. This shows $\alpha\leq d(A)$. In order to complete
the proof of the proposition we need to see that if $\alpha=d(A)$
then $s<n_{A}$. To this end, recall that $s$ is the maximal number
of alternating sets of cardinality $\alpha+1$ in nonidentities (in
addition to arbitrary many alternating sets of cardinality $\alpha$).
But if $\alpha=d(A)$, then alternating sets of cardinality $\alpha+1$
must contain at least one radical evaluation on any nonzero evaluation
of its variables and hence, as above, the polynomial cannot contain
more than $(n_{A}-1)$ alternating sets of cardinality $\alpha+1$.
This proves the proposition.
\end{proof}
In order to establish a precise relation between the index of a finite
dimensional $H$-module algebra $A$ and its structure we need to
find appropriate finite dimensional $H$-module algebras which will
serve as a \textbf{minimal model} for a given $H$-Kemer index. Here
is the precise definition.
\begin{defn}
A finite dimensional $H$-module algebra $A$ is said to be \textit{$H$-PI}-\textit{basic}
(or just \textit{$H$-basic}) if there are no finite dimensional $H$-module
algebras $B_{1},...,B_{s}$ such that $\mbox{Par}(B_{i})<\mbox{Par}(A)$
and $A$ is $H$-PI equivalent to $B_{1}\times\cdots\times B_{s}$.\end{defn}
\begin{rem}
\label{rem:basic_is_everywhere}By induction on $\mbox{Par}(A)$ it
is easy to see that every finite dimensional $H$-module algebra is
$H$-PI equivalent to a finite product of $H$-basic algebras.
\end{rem}
We need to understand what ``PI properties'' does $H$-basic algebras
posses. 
\begin{defn}
We say that a finite dimensional $H$-module algebra $A$ is \textit{full}
with respect to a multilinear $H$-polynomial $f$, if exist a nonvanishing
evaluation of $f$ on $A$ such that every $H$-simple component is
represented (among the semisimple substitutions). A finite dimensional
$H$-module algebra $A$ is said to be full if it is full with respect
to some multilinear $H$-polynomial $f$. \end{defn}
\begin{lem}
\label{lem:full}Let $A$ be a finite dimensional $H$-module algebra
which is not full. Then $A$ is not $H$-basic.\end{lem}
\begin{proof}
Since any $H$-module algebra with one $H$-simple component is full
we may assume that $q>1$. Consider the decompositions mentioned above
$A\cong\overline{A}\oplus J$ and $\overline{A}\cong A_{1}\times A_{2}\times\cdots\times A_{q}$
($A_{i}$ are $H$-simple algebras). Construct the $H$-module subalgebras
\[
B_{i}=\left(A_{1}\times\cdots\times A_{i-1}\times A_{i+1}\times\cdots\times A_{q}\right)\oplus J=\pi^{-1}\left(A_{1}\times\cdots\times A_{i-1}\times A_{i+1}\times\cdots\times A_{q}\right),
\]
where $\pi:A\to\bar{A}$ is the natural projection. 

We claim that the algebras $A$ and $\widetilde{A}=B_{1}\times\cdots\times B_{q}$
are \textit{$H$-PI}-equivalent: Of course $id^{H}(A)\subseteq id^{H}(\widetilde{A})$,
so it suffices to prove that any $H$-nonidentity $f$ of $A$ is
also a nonidentity of $\widetilde{A}$. Clearly, we may assume that
$f$ is multilinear (say of degree $n$). Consider a non zero evaluation
$\bar{x}_{1},...,\bar{x}_{n}$ of $f$ on $A$. By assumption, there
is some $i$ such that $\bar{x}_{1},...,\bar{x}_{n}\notin A_{i}$
so $\bar{x}_{1},...,\bar{x}_{n}\in B_{i}$. Hence $f$ is non zero
on $\widetilde{A}$. Since for every $i$ $\mbox{Par}(B_{i})<\mbox{Par}(A)$
we are done.\end{proof}
\begin{prop}
\label{prop:unique_point}Let $A$ be a finite dimensional $H$-module
algebra which is full. Let $\mbox{Ind}(A)=(\alpha,s)$ and $\mbox{Par}(A)=(d(A),n_{A}-1)$.
Then $\alpha=d(A)$.
\end{prop}
For the proof we need to show that for an arbitrary large integer
$\mu$ there exists a multilinear $H$-nonidentity $f$ that contains
$\mu$ folds of alternating sets of cardinality $\dim_{F}(\overline{A})$.
\begin{lem}[Kemer's Lemma $1$]
\label{lem:full-folds}Notation as above. Let $A$ be a finite $H$-module
dimensional algebra which is full. Then for any integer $\mu$ there
exists a polynomial $f$ in the $T$-ideal with the following properties: 
\begin{enumerate}
\item $f\notin id^{H}(A)$ 
\item $f$ has $\mu$-folds of alternating sets of cardinality $\dim_{F}(\overline{A})$.
\end{enumerate}
\end{lem}
\begin{proof}
See Lemma 10 in \cite{Gordienko2013}.
\end{proof}

\section{Kemer's Lemma $2$\label{Section: Kemer's_Lemma_=0000242=000024}}

In this section we prove Kemer's Lemma $2$. Before stating the precise
statement we need to extract an additional ``PI property'' from
$H$-basic algebras. This time we need a property which controls the
nilpotency index.

Let $f$ be a multilinear $H$-polynomial which is not in $id^{H}(A)$.
Clearly, any nonzero evaluation cannot have more than $n_{A}-1$ radical
evaluations.
\begin{lem}
Let $A$ be a finite dimensional $H$-module algebra. Let $\mbox{Ind}(A)=(\alpha,s)$
be its Kemer index. Then $s\leq n_{A}-1$.\end{lem}
\begin{proof}
$A$ is $H$-PI equivalent to the direct product of $H$-module algebras
$B_{1}\times\cdots\times B_{q}$, where $B_{i}$ is full for $i=1...q$.
For each $B_{i}$ we consider the dimension of the semisimple part
$d(B_{i})$. Applying Kemer lemma $1$ we have that $\alpha\geq\max_{i}(d(B_{i}))$.
On the other hand if $\alpha>d(B_{i})$, any multilinear polynomial
with more than $n_{B_{i}}-1$ alternating sets of cardinality $\alpha$
is in $id^{H}(B_{i})$ (any alternating set must have at least one
radical evaluation) and hence if $\alpha>\max_{i}(d(B_{i}))$, any
polynomial as above is an identity of $B_{1}\times\cdots\times B_{q}$
and hence of $A$. This contradicts the definition of the parameter
$\alpha$ and hence $\alpha=\max_{i}(d(B_{i}))$. Now take an alternating
set of cardinality $\alpha+1$. In every such set we must have a radical
evaluation or elements from different full algebras. If they come
from different full algebras we get zero. If we get a radical element
then we cannot pass $n_{A}-1$. 
\end{proof}
The next definition is key in the proof of Kemer's Lemma $2$ (see
below). 
\begin{defn}
Notation as above. Let $f$ be a multilinear polynomial which is not
in $id^{H}(A)$. We say that $A$ has property $K$ with respect to
$f$ if $f$ vanishes on any evaluation on $A$ with less than $n_{A}-1$
radical substitutions.

We say that a finite dimensional $H$-module algebra $A$ has property
$K$ if it satisfies the property with respect to some nonidentity
multilinear $H$-polynomial.\end{defn}
\begin{prop}
\label{prop:property_K}Let $A$ be $H$-basic algebra. Then it has
property $K$. Moreover there is a multilinear $H$-polynomial which
satisfies property $K$ and is full.
\end{prop}
Before proving the proposition we introduce a construction which will
enable us to put some ``control'' on the nilpotency index of (the
radical of) finite dimensional $H$-module algebras which are $H$-PI
equivalent.

Let $B$ be any finite dimensional $H$-module algebra and let $B^{'}=\overline{B}\ast F^{H}\left\langle x_{1},\dots,x_{t}\right\rangle $
be the co-product of the Free $H$-module algebra on the generators
$\{x_{1},...,x_{t}\}$ with the algebra $\overline{B}$, the semisimple
component of $B$. We define an $H$ action in the following fashion
\[
h\cdot b_{1}f_{1}\cdots f_{k}b_{k+1}=h_{(1)}(b_{1})h_{(2)}(f_{1})\cdots h_{(2k)}(f_{k})h_{(2k+1)}(b_{k+1})
\]
where $b_{1},...,b_{k}\in\overline{B}$ and $f_{1},...,f_{k}\in F^{H}\left\langle x_{1},\dots,x_{t}\right\rangle $.
The number of variables we take is at least the dimension of $J(B)$.
Let $I_{1}$ be the $H$-ideal of $B^{'}$ generated by all evaluations
of polynomials of $id^{H}(B)$ on $B^{'}$ and let $I_{2}$ be the
$H$-ideal generated by all variables $x_{i}^{h}$, where $h\in H$.
Consider the $H$-module algebra $\widehat{B}_{u}=B^{'}/(I_{1}+I_{2}^{u})$.
\begin{prop}
\label{prop:Control_on_nilpotency_index}The following hold:
\begin{enumerate}
\item $id^{H}(\widehat{B}_{u})=id^{H}(B)$ whenever $u\geq n_{B}$ $($$n_{B}$
denotes the nilpotency index of $J(B)$$)$. In particular $\widehat{B}_{u}$
and $B$ have the same index.
\item $\widehat{B}_{u}$ is finite dimensional.
\item The nilpotency index of $J(\widehat{B}_{u})$ is $\leq u$.
\end{enumerate}
\end{prop}
\begin{proof}
Note that by the definition of $\widehat{B}_{u}$ (modding $B^{'}$
by the ideal $I_{1}$), $id^{H}(\widehat{B}_{u})\supseteq id^{H}(B)$.
On the other hand there is a surjection $\phi:\widehat{B}_{u}\longrightarrow B$
which maps the variables $\{x_{i}\}$ onto a spanning set of $J(B)$
and $\overline{B}$ is mapped isomorphically. The ideal $I_{1}$ consist
of all evaluation of $id^{H}(B)$ on $B^{'}$ and hence is contained
in $ker(\phi)$. Also the ideal $I_{2}^{u}$ is contained in $ker(\phi)$
since $u\geq n_{B}$ and $\phi(x)\in J$. This shows (1).

To see (2) observe that any element in $\widehat{B}_{u}$ is represented
by a sum of elements the form $b_{1}z_{1}^{g_{1}}b_{2}z_{2}^{g_{2}}\cdots b_{j}z_{j}^{g}b_{j+1}$
where $j<u$, $b_{i}\in\overline{B}$, $z_{i}\in\{x_{1},...,x_{t}\}$
and $g_{i}$ is in a basis of $H$. In order to prove the 3rd statement,
note that $I_{2}$ generates a radical ideal $\widehat{B}_{u}$ and
since $B^{'}/I_{2}\cong\overline{B}$ we have that 
\[
\widehat{B}_{u}/I_{2}\cong B^{'}/(I_{1}+I_{2}^{u}+I_{2})=B^{'}/(I_{1}+I_{2})\cong(B^{'}/(I_{2}))/I_{1}=\overline{B}/I_{1}=\overline{B}
\]
 (the last equality follows from the fact that $\overline{B}\subseteq B$).
We therefore see that $I_{2}$ generates the radical in $\widehat{B}_{u}$,
and hence its nilpotency index is bounded by $u$ as claimed.
\begin{proof}
(of Proposition \ref{prop:property_K}) Let $B_{1},...,B_{q}$ be
the $H$-module algebras defined in \lemref{full} and consider the
$H$-module algebra $\widehat{A}_{u}=A^{'}/(I_{1}+I_{2}^{u})$ (from
the proposition above). It is clear that $id^{H}(\widehat{A}_{n_{A}-1}\times B_{1}\times\cdots\times B_{q})\supsetneq id^{H}(A)$.
We show that if the proposition is false (for $A$), then there is
an equality. Since $\mbox{Par}(B_{i}),\mbox{Par}(\widehat{A}_{n_{A}-1})<\mbox{Par}(A)$
we get a contradiction.

Take a multilinear polynomial $f=f(x_{1},...,x_{n})$ which is not
in $id^{H}(A)$ and consider a non zero evaluation $\bar{x}_{1},...,\bar{x}_{n}$
on $A$. Suppose that $\bar{x}_{1},...,\bar{x}_{v}\in J$ and the
rest are in $\overline{A}$. If $v<n_{A}$, then $f^{\prime}=f(x_{1},...,x_{s},\bar{x}_{v+1},...,\bar{x}_{n})\in B^{'}$
and is non zero in $\widehat{A}_{n_{A}-1}$, since otherwise $f^{\prime}=\sum_{i}g_{i}(x_{1},...,x_{v},Y)$,
where $g_{i}\in id^{H}(A)$ and $Y\subseteq\overline{A}$. So by substituting
$\bar{x}_{i}$ instead $x_{i}$, we will get that $f(\bar{x}_{1},...,\bar{x}_{n})=0$.
If not all the simple components of $A$ appear in $\bar{x}_{v+1},...,\bar{x}_{n}$
we get (see the proof of \lemref{full}) that $f$ is a non identity
of one of the $B_{i}$. By our assumption these are the only options.
Hence, in any case $f$ is a non identity of the product $\widehat{A}_{n_{A}-1}\times B_{1}\times\cdots\times B_{q}$
as claimed. 
\end{proof}
\end{proof}
In the next lemma we deal with properties which are preserved in $H$-$T$-ideals.
\begin{lem}
\label{lem:Phoenix}Let $A$ be $H$-basic. The following hold.
\begin{enumerate}
\item Let $f\notin id^{H}(A)$ be a multilinear polynomial and suppose $A$
is full with respect to nonzero evaluations of $f$ on $A$, that
is, in any nonzero evaluation of $f$ on $A$ we must have semisimple
values from all $H$-simple components. Then if $f'\in\langle f\rangle_{H}$
is multilinear ($\langle f\rangle_{H}=$ $H$-$T$-ideal generated
by $f$$)$ is a nonidentity of $A$ then it is full with respect
to any nonzero evaluation on $A$.
\item Let $f\notin id^{H}(A)$ be multilinear and suppose it is $\mu$-fold
alternating on disjoint sets of cardinality $d(A)=\dim_{F}(\overline{A})$.
If $f^{'}\in\langle f\rangle_{H}$ is a nonidentity of $A$, then
there exists a nonidentity $f^{''}\in\langle f^{'}\rangle_{H}$ of
$A$, which is multilinear and $\mu$-fold alternating on sets of
cardinality $d(A)$. In other words, the property of being $\mu$-fold
alternating on sets of cardinality $d(A)$ is $A$-Phoenix.
\item Property $K$ is strictly $A$-Phoenix.
\end{enumerate}
\end{lem}
\begin{proof}
Suppose $f(x_{1},\ldots,x_{n})$ is a multilinear polynomial which
satisfies the condition in $1$. It is sufficient to show the condition
remains valid if $f'$ is multilinear and has the form (a) $f'=\sum_{i}p_{i}\cdot f\cdot q_{i}$
(b) $f'(z_{1},\ldots,z_{t},x_{2},\ldots,x_{n})=f(Z,x_{2},\ldots,x_{n})$
where $Z=z_{1}^{h_{1}}\cdots z_{t}^{h_{t}}$ is a multilinear monomial
consisting of variables disjoint to the variables of $f(x_{1},\ldots,x_{n})$.
If $f'=\sum_{i}p_{i}\cdot f\cdot q_{i}$ then any nonzero evaluation
of $f'$ arises from a nonzero evaluation of $f$ and so the claim
is clear in this case. Let $f'(z_{1},\ldots,z_{t},x_{2},\ldots,x_{n})=f(Z,x_{2},\ldots,x_{n})$
and suppose $x_{i}=\hat{x}_{i}$ and $z_{i}=\hat{z}_{i}$ is a non
vanishing evaluation of $f'$. If an $H$-simple component $A_{1}$
say, is not represented, then the same simple component is not represented
in the evaluation $x_{1}=\hat{z}_{1}^{h_{1}}\cdots\hat{z}_{t}^{h_{t}},x_{2}=\hat{x}_{2},\ldots,x_{n}=\hat{x}_{n}$
and hence $f$ vanishes. We see that $f'$ vanishes on any evaluation
which misses a simple component.

For the second part of the lemma note that if $f$ is multilinear
and has $\mu$-folds of alternating sets of cardinality $d(A)=\dim_{F}(\overline{A})$
then clearly it vanishes on any evaluation unless it visits in all
simple components and hence the result follows from the first part
of the Lemma and Kemer Lemma 1.

We now turn to the proof of the $3$rd part of the lemma. If $f'=\sum_{i}g_{i}fp_{i}$
then it is clear that if an evaluation of $f'$ has less than $n_{A}-1$
radical evaluations then with that evaluation $f$ has less than $n_{A}-1$
radical evaluations and hence vanishes. This implies the vanishing
of $f'$. If an evaluation of $f'(z_{1},\ldots,z_{t},x_{2},\ldots,x_{n})=f(Z,x_{2},\ldots,x_{n})$
has less than $n_{A}-1$ radicals, then this corresponds to an evaluation
of $f(x_{1},\ldots,x_{n})$ with less than $n_{A}-1$ radicals and
hence vanishes.
\end{proof}
We can now state and prove Kemer's lemma $2$.
\begin{lem}[Kemer's lemma $2$]
\label{lem:Kemer's_Lemma_2} Let $A$ be $H$-full algebra. Suppose
$\mbox{Par}(A)=(d=d(A),n_{A}-1)$. Then for any integer $\nu$ there
exists a multilinear nonidentity $f$ with $\mu$-alternating sets
of cardinality $d$ $($small sets$)$ and precisely $n_{A}-1$ alternating
sets of variables of cardinality $d+1$ $($big sets$)$.\end{lem}
\begin{rem}
The theorem is clear either in case $A$ is radical or semisimple
(i.e. $H$-simple). Hence for the proof we assume that $q\geq1$ (the
number of simple components of $A$) and $n_{A}>1$. 
\begin{rem}
Any nonzero evaluation of such $f$ must consists only of semisimple
evaluations in the $\nu$-folds and each one of the big sets (namely
the sets of cardinality $d+1$) must have exactly one radical evaluation.
\end{rem}
\end{rem}
\begin{proof}
By the preceding Lemma we take a multilinear nonidentity $H$-polynomial
$f$, with respect to which $A$ is full and has property $K$. Let
us fix a nonzero evaluation $x\longmapsto\widehat{x}$ realizing the
``full'' property. Note (by the remark above) that by the construction
of $f$, being the evaluation nonzero, precisely $n_{A}-1$ variables
must obtain radical values and the rest of the variables obtain semisimple
values. Let us denote by $w_{1},\ldots,w_{n_{A}-1}$ the variables
that obtain radical values (in the evaluation above) and by $\widehat{w}_{1},\ldots,\widehat{w}_{n_{A}-1}$
their corresponding values. By abuse of language we refer to the variables
$w_{1},\ldots,w_{n_{A}-1}$ as \textit{radical} variables.

We will consider four cases. These correspond to whether $A$ has
or does not have an identity element and whether $q$ (the number
of $H-$simple components) $>1$ or $q=1$.

\textbf{Case $(1,1)$} ($A$ has an identity element and $q>1$).

By linearity we may assume the evaluation of any radical variable
$w_{i}$ is of the form $1_{A_{j(i)}}\widehat{w}_{i}1_{A_{\widetilde{j}(i)}}$,
$i=1,\ldots,n_{A}-1$, where $1_{A_{k}}$ is the identity element
of the $H$-simple component $A_{k}$. Note that the evaluation remains
full (i.e. visits any simple component of $A$).

Choose a monomial $X$ of $f$ which does not vanish upon the above
evaluation. Notice that the variables of $X$ which get semisimple
evaluations from \textit{different} $H$-simple components must be
separated by radical variables.

Consider the radical evaluations which are bordered by pairs of elements
$(1_{A_{j(i)}},1_{A_{\widetilde{j}(i)}})$ where $j(i)\neq\widetilde{j}(i)$
(i.e. belong to different $H$-simple components). Then it is clear
that every simple component is represented by one of the elements
in these pairs.

For $t=1,\ldots,q$ we fix a variable $w_{r_{t}}$ whose radical value
is $1_{A_{j(r_{t})}}\widehat{w}_{r_{t}}1_{A_{\widetilde{j}(r_{t})}}$
where
\begin{enumerate}
\item $j(r_{t})\neq\widetilde{j}(r_{t})$ (i.e. different $H$-simple components).
\item One of the element $1_{A_{j(r_{t})}},1_{A_{\widetilde{j}(r_{t})}}$
is the identity element of the $t$-th simple component.
\end{enumerate}
We refer to that element as the \textit{idempotent attached} to the
simple component $A_{t}$.
\begin{rem*}
Note that we may have $w_{r_{t}}=w_{r_{t^{'}}}$ even if $t\neq t^{'}$.
\end{rem*}
Next replace the variables $w_{r_{t}}$, $t=1,\ldots,q$ by $z_{r_{t}}y_{r_{t}}z_{r_{t}}^{\prime}w_{r_{t}}$
or $w_{r_{t}}z_{r_{t}}y_{r_{t}}z_{r_{t}}^{\prime}$ (and obtain a
new polynomial $f_{1}$) according to the location of the primitive
$H$-invariant idempotent attached to the $t$-th simple component.
Clearly, by evaluating the variables $y_{r_{t}},z_{r_{t}}$ and $z_{r_{t}}^{\prime}$
by $1_{A_{j(r_{t})}}$ (or $1_{A_{\tilde{j}(r_{t})}}$) the value
of $f_{1}$ remains the same as $f_{1}$ under the original substitution
and in particular nonzero. For later reference we call the variables
$z_{r_{t}}$ and $z_{r_{t}}^{\prime}$ \emph{frame} \emph{variables
}and consider the evaluation $1_{A_{j(r_{t})}}\to z_{r_{t}},z_{r_{t}}^{\prime}$
(or $1_{A_{\tilde{j}(r_{t})}}$).

Applying \lemref{full-folds} we can replace (in $f_{1}$) the variable
$y_{r_{t}}$, $t=1,\dots,q$, by a $\mu$-fold alternating polynomial
(on the distinct sets $U_{l}^{t}$) $Z_{r_{t}}=Z_{r_{t}}(U_{1}^{t},\ldots,U_{\nu}^{t};Y_{t})$,
and obtain a nonzero polynomial $f_{2}$. Here, the sets $U_{l}^{t}$,
$l=1,\ldots,\nu$ are each of cardinality $\dim_{F}(A_{t})$. Now,
if we further alternate the sets $U_{l}^{1},\ldots,U_{l}^{q}$ for
$l=1,\ldots,\nu$ together (that is for each $l$, apply $\mbox{Alt}_{U_{l}^{1}\cup\cdots\cup U_{l}^{t}}$)
we obtain a nonidentity polynomial with $\nu$-folds of (small) sets
of alternating variables where each set is of cardinality $\dim(\overline{A})$.
In the sequel we fix an evaluation of the polynomials $Z_{r_{t}}$
(or $\widetilde{Z}_{r_{t}}$) so the entire polynomial obtains a nonzero
value. 

Our next task is to construct such polynomial with an extra $n_{A}-1$
alternating sets of cardinality $d+1$ (big sets). Consider the radical
variables $w_{r_{t}}$, $t=1,\dots,q$ with radical evaluations $1_{A_{j(r_{t})}}\widehat{w}_{r_{t}}1_{A_{\widetilde{j}(r_{t})}}$,
$j(r_{t})\neq\widetilde{j}(r_{t})$ (i.e. different $H$-simple components).

We attach each variable $w_{r_{t}}$ to one alternating set $U_{l}^{1},\ldots,U_{l}^{q}$
(some $l$). We see that any nontrivial permutation of $w_{r_{t}}$
with one of the variables of $U_{l}^{1},\ldots,U_{l}^{q}$, keeping
the evaluation above, will yield a \textit{zero value} since the primitive
$H$-invariant idempotents values in frames variables of each $Z_{r_{1}},\ldots,Z_{r_{q}}$
belong to the same $H$-simple components whereas the pair of idempotents
in $1_{A_{j(r_{t})}}\widehat{w}_{r_{t}}1_{A_{\widetilde{j}(r_{t})}}$
belong to different $H$-simple components. Thus we may alternate
the variable $w_{r_{t}}$ with $U_{l_{t}}^{1},\ldots,U_{l_{t}}^{q}$,
$t=1,\ldots,q$ and obtain a multilinear nonidentity of $A$. Next
we proceed in a similar way with any remainig variable $w_{i}$ whose
evaluation is $1_{A_{j(i)}}\widehat{w}_{i}1_{A_{\widetilde{j}(i)}}$
and $j(i)\neq\widetilde{j}(i)$.

Finally we need to attach the radical variables $w_{i}$ whose evaluation
is $1_{A_{j(i)}}\widehat{w}_{i}1_{A_{\widetilde{j}(i)}}$ where $j(i)=\widetilde{j}(i)$
(i.e. the same simple component) to some small sets. We claim also
here that if we attach the variable $w_{i}$ to the sets $U_{l}^{1},\ldots,U_{l}^{q}$
(some $l$), any nontrivial permutation yields a zero value, and hence
the value of the entire polynomial remains unchanged. If we permute
$w_{i}$ with an element $u_{0}\in U_{l}^{k}$ which is bordered by
idempotents different from $1_{A_{j(i)}}$ we obtain zero. On the
other we claim that any permutation of $w_{i}$ with an element $u_{0}\in U_{l}^{k}$
which is bordered by the idempotent $1_{A_{j(i)}}$ corresponds to
an evaluation of the original polynomial with fewer radical values
and then we will be done by the property $K$. In order to simplify
our notation let $\{U_{l}^{1},\ldots,U_{l}^{q}\}=\{U^{1},\ldots,U^{q}\}$
(omit the index $l$) and suppose without loss of generality, that
$u_{0}\in U^{1}$. Permuting the variables $w_{i}$ and $u_{0}$ (with
their corresponding evaluations) we see that the polynomial $Z_{r_{1}}=Z_{r_{1}}(U^{1}=U_{1}^{1},\ldots,U_{\nu}^{t};Y_{t})$
(or $\widetilde{Z}_{r_{1}}$) with $w_{i}$ replacing $u_{0}$, obtains
a radical value which we denote by $\widehat{\widehat{w}}$. Returning
to our original polynomial $f$, we obtain the same value if we evaluate
the variable $w_{i}$ by a suitable semisimple element, the variable
$w_{r_{1}}$ by $\widehat{\widehat{w}}\widehat{w}_{r_{1}}$ (or $\widehat{w}_{r_{1}}\widehat{\widehat{w}}$)
and the evaluation of any semisimple variable remains semisimple.
It follows that if we make such a permutation for a unique radical
variable $w_{i}$, the value amounts to an evaluation of the original
polynomial with $n_{A}-2$ radical evaluations and hence vanishes.
Clearly, composing $p>0$ permutations of that kind yields a value
which may be obtained by the original polynomial $f$ with $n_{A}-1-p$
radical evaluations and hence vanishes by property $K$. This completes
the proof of the lemma where $A$ has identity and $q$, the number
of simple components, is $>1$.

\smallskip{}

\textbf{Case $(2,1)$}. Suppose now $A$ has no identity element and
$q>1$. Let $A_{0}=A\oplus F1$, where $h(1)=\epsilon(h)1$. The proof
in this case is basically the same as in the case where $A$ has an
identity element. Let $e_{0}=1-1_{A_{1}}-1_{A_{2}}-\cdots-1_{A_{q}}\in A_{0}$
and attach $e_{0}$ to the set of elements which border the radical
values $\widehat{w}_{j}$. A similar argument shows that also here
every $H$-simple component ($A_{1},\ldots,A_{q}$) is represented
in one of the bordering pairs where the partners are \textit{different}
(the point is that one of the partners (among these pairs) may be
$e_{0}$). Now we complete the proof exactly as in case $(1,1)$.

\smallskip{}

\textbf{Case $(2,2)$}. In order to complete the proof of the lemma
we consider the case where $A$ has no identity element and $q=1$.
The argument in this case is different. For simplicity we denote by
$e_{1}=1_{A_{1}}$ and $e_{0}=1-e_{1}$. Let $f(x_{1},\ldots,x_{n})$
be a nonidentity of $A$ which satisfies property $K$ and let $f(\widehat{x}_{1},\ldots,\widehat{x}_{n})$
be a nonzero evaluation for which $A$ is full. If $e_{1}f(\widehat{x}_{1},\ldots,\widehat{x}_{n})\neq0$
(or $f(\widehat{x}_{1},\ldots,\widehat{x}_{n})e_{1}$) we proceed
as in case $(1,2)$. To treat the remaining case we may assume further
that

\[
e_{0}f(\widehat{x}_{1},\ldots,\widehat{x}_{n})e_{0}\neq0
\]

First note, by linearity, that each one of the radical values $\widehat{w}$
may be bordered by one of the pairs $\{(e_{0},e_{0}),(e_{0},e_{1}),(e_{1},e_{0}),(e_{1},e_{1})\}$
so that if we replace the evaluation $\widehat{w}$ (of $w$) by the
corresponding element $e_{i}\widehat{w}e_{j}$, $i,j=0,1$, we get
nonzero.

Now, one of the radical values (say $\widehat{w_{0}}$) in $f(\widehat{x}_{1},\ldots,\widehat{x}_{n}$)
allows a bordering by the pair $(e_{0},e_{1})$ (or $(e_{1},e_{0})$),
then replacing $w_{0}$ by $w_{0}y$ (or $yw_{0}$) yields a nonidentity.
Invoking Lemma \ref{lem:full-folds} we may replace the variable $y$
by a polynomial $p$ with $\mu$-folds of alternating (small) sets
of cardinality $\dim_{F}(\overline{A})=\dim_{F}(A_{1})$. Then we
attach the radical variable $w_{0}$ to one of the small sets. Clearly,
the value of any alternation of this (big) set is zero since the borderings
are different. The remaining possible values of radical variables
are either $e_{0}\widehat{w}e_{0}$ or $e_{1}\widehat{w}e_{1}$. Note
that since semisimple values can be bordered only by the pair $(e_{1},e_{1})$,
any alternation of the radical variables whose radical value is $e_{0}\widehat{w}e_{0}$
with elements of a small set vanishes and again the value of the polynomial
remains unchanged. Finally we attach the remaining radical variables
(whose values are to suitable small sets in $p$. Here, any alternation
vanishes because of property $K$. This settles this case. Obviously,
the same holds if the bordering pair above is $(e_{1},e_{0})$. \end{proof}
\begin{cor}
If $A$ is basic then its $H$-Kemer index $(\alpha,s)$ equals $(d,n_{A}-1)$.
\begin{cor}
\label{Kemer_polynomials_of_A_are_Phoenix}Let $A$ be a finite dimensional
$H$-module algebra, then there is a number $\mu_{A}^{\prime}$ such
that every $H$-Kemer polynomial $f$ of $A$ of rank at least $\mu_{A}^{\prime}$
satisfies the $A$-Phoenix property.
\end{cor}
\end{cor}
\begin{proof}
Suppose $A$ is $H$-basic. Clearly if $f$ is $H$-Kemer of rank
$\mu\geq\mu_{A}=\mu_{A}^{\prime}$ then $A$ is full and satisfies
property $K$ with respect to $f$. The Corollary now follows from
Lemmas \lemref{Phoenix} and \lemref{Kemer's_Lemma_2}. 

Consider now the general case, we may suppose that $A=B_{1}\times\cdots\times B_{q}$
is a product of $H$-basic algebras.Let $f^{\prime}$be a multilinear
consequence of $f$ which is not an identity of $A$. Thus $f^{\prime}$
must be a non identity of (at least) one of the $B_{i}$, say $B_{1}$.
Therefore, $f$ is also a non identity of $B_{1}$. Thus, if $\mu\geq\mu_{B_{1}}$
we can conclude that $\mbox{Ind}(A)\leq\mbox{Ind}(B_{1})$, so $\mbox{Ind}(A)=\mbox{Ind}(B_{1})$.
Hence $f$ is $H$-Kemer of $B_{1}$. By the previous paragraph we
are done if we set $\mu_{A}^{\prime}=\max\{\mu_{B_{1}},...,\mu_{B_{q}}\}$.
\end{proof}

\section{Technical tools\label{sec:technical}}

\subsection{Affine relatively $H$-module algebras \label{sub:Affine-relatively--module}}

Recall that an algebra $W$ satisfies the $t$th Capelli identity
if any multilinear polynomial having an alternating set of cardinality
(at least) $t$ is an $H$-identity of $W$. The purpose of this section
is to prove that for any such algebra one can assume that the corresponding
relatively free algebra $\mathcal{W}$ is generated by (only) $t-1$
variables. More precisely, we will show that if 
\[
\mathcal{W}=F^{H}\left\langle x_{1},...,x_{t-1}\right\rangle /id^{H}(W)\cap F^{m}\left\langle x_{1},...,x_{t-1}\right\rangle 
\]
then $id^{H}(W)=id^{H}(\mathcal{W})$. To this end we recall some
basic results (and fix notation) from the representation theory of
$S_{n}$ (the symmetric group on $n$ elements) and their application
to PI theory.

Let $P_{n}^{H}(W)=P_{n}^{H}/(P_{n}\cap id^{H}(W))$\nomenclature[993]{$P_{n}^{H}(W)$}{The space $P_{n}^{H}/(P_{n}\cap id^{H}(W))$},
where $P_{n}^{H}$ is the space (of dimension $\left(\dim_{F}H\right)^{n}\cdot n!$)
of all multilinear polynomials with variables $x_{1},...,x_{n}$.
The group $S_{n}$ acts on (right action!) $P_{n}^{H}(W)$ via $\sigma\cdot\overline{x_{i_{1}}^{h_{1}}\cdots x_{i_{n}}^{h_{n}}}=\overline{x_{\sigma(i_{1})}^{h_{1}}\cdots x_{\sigma(i_{n})}^{h_{n}}}$
and hence we may consider its decomposition into irreducible submodules.
By the representation theory of $S_{n}$\nomenclature[994]{$S_{n}$}{The symmetric group of $n$ elements}
in characteristic zero, any such submodule can be written as $FS_{n}e_{T_{\mu}}\cdot f$,
where $f$ is some polynomial in $P_{n}^{H}(W)$, $T_{\mu}$\nomenclature[995]{$T_{\mu}$}{A Young tableau corresponding to the partition $\mu$}
is some Young tableau of the partition $\mu$ (of $n$) and 
\[
e_{T_{\mu}}=\sum_{\sigma\in\mathcal{R}_{T_{\mu}},\tau\in\mathcal{C}_{T_{\mu}}}(-1)^{\tau}\sigma\tau.
\]
\nomenclature[997]{$e_{T_{\mu}}$}{The  $FS_n$ element $\sum_{\sigma\in\mathcal{R}_{T_{\mu}},\tau\in\mathcal{C}_{T_{\mu}}}(-1)^{\tau}\sigma\tau$}(here
$\mathcal{R}_{T_{\mu}}$ and $\mathcal{C}_{T_{\mu}}$\nomenclature[996]{$\mathcal{R}_{T_{\mu}}$ / $\mathcal{C}_{T_{\mu}}$}{The row and column stabilizers of the young tableau $T_\mu$}
are the rows and columns stabilizers respectively). Clearly, if $f\in P_{n}^{H}(W)$
is nonzero, then there is some partition $\mu$ and a (standard) tableau
$T_{\mu}$such that $e_{T_{\mu}}\cdot f$ is nonzero.

We are ready to prove the main result of this section.
\begin{thm}
\label{thm:relativelly_affine}Let $W$ be an $H$-module algebra
which satisfies the $t$th Capelli identity. Then $id^{H}(\mathcal{W})=id^{H}(W)$
where $\mathcal{W}$ is the relatively free $H$-module algebra of
$W$ generated by $t-1$ variables. \end{thm}
\begin{proof}
It is clear that $id^{H}(W)\subset id^{H}(\mathcal{W})$. For the
other direction suppose $f$ is a multilinear nonidentity of $W$
of degree $n$. Then, by the theorem above, there is a partition $\mu$
of $n$ and a tableau $T_{\mu}$ such that $g=e_{T_{\mu}}\cdot f$
is a nonidentity of $W$.

Let $g_{0}=\sum_{\tau\in\mathcal{C}_{T_{\mu}}}(-1)^{\tau}\tau\cdot f=\sum_{\tau\in\mathcal{C}_{T_{\mu}}(1)}(-1)^{\tau}\tau\cdot\left(\sum_{k=1}^{l}(-1)^{\tau_{k}}\tau_{k}\cdot f\right)$,
where $\mathcal{C}_{T_{\mu}}(1)$ is the stabilizer of the first column
of $T_{\mu}$ and $\tau_{1},...,\tau_{l}$ is a full set of representatives
of $\mathcal{C}_{T_{\mu}}(1)$-cosets in $C_{T_{\mu}}$.

Let $h(\mu)$\nomenclature[998]{$h(\mu)$}{The height of $\mu$} (the
height of $\mu$) denote the number of rows in $T_{\mu}$. If $h(\mu)\geq t$,
the polynomial $g_{0}$ is alternating on the variables of the first
column and hence by assumption is an identity of $W$. But in that
case also the polynomial $g=\sum_{\sigma\in R_{T_{\mu}}}\sigma\cdot g_{0}$
is in $id^{H}(W)$ contradicting our assumption and so $h(\mu)$ must
be smaller than $t$.

Since $g=\sum_{\sigma\in\mathcal{R}_{T_{\mu}}}\sigma\cdot g_{0}$,
it is symmetric in the variables corresponding to any row of $T_{\mu}$
and so if for any $i=1,\ldots,h(\mu)$ we replace by $y_{i}$ \textit{all}
variables in $g$ corresponding to the $i$th row we obtain a polynomial
$\hat{g}$ which yields $g$ by multinearization. In particular $g\in id^{H}(W)$
if and only if $\hat{g}\in id^{H}(W)$. Finally, $\hat{g}$ can be
regarded as an element of $\mathcal{W}$ (at most $t-1$ variables)
and nonzero, thus $g$ is a nonidentity of $\mathcal{W}$ and hence
also $f$. \end{proof}
\begin{rem}
In the sequel, if $W$ satisfies the $t$th Capelli identity, we'll
consider affine relatively free $H$-module algebras $\mathcal{W}$
with at least $t-1$ generating variables.\end{rem}
\begin{defn}
Suppose $W$ is an affine $H$-module algebra. Any algebra of the
form 
\[
F^{H}\left\langle x_{1},...,x_{t}\right\rangle /id^{H}(W)\cap F^{H}\left\langle x_{1},...,x_{t}\right\rangle 
\]
having the same $T$ ideal as $W$ is called \emph{affine relatively
free $H$-module algebra of $W$}. 
\end{defn}
We close this subsection with the following useful lemma.
\begin{cor}
\label{cor:useful_lemma} Suppose $\mathcal{W}$ is an relatively
free $H$-module algebra of $W$ (in particular we will be interested
in the case where $\mathcal{W}$ is affine). Let $I$ be any $H$-$T$
ideal and denote by $\widehat{I}$ the ideal of $\mathcal{W}$ generated
(or consisting rather) by all evaluation on $\mathcal{W}$ of elements
of $I$. Then $id^{H}(\mathcal{W}/\widehat{I})=id^{H}(W)+I$.
\end{cor}

\subsection{Shirshov base\label{sub:Shirshov-base}}
\begin{defn}
Let $W$ be an affine algebra over $F$. Let $a_{1},...,a_{s}$ be
a generating set of $W$. Let $t$ be a positive integer and let $Y$
be the set of words in $a_{1},...,a_{s}$ of length $\leq t$. We
say that $W$ has a Shirshov base of length $t$ and of height $h$
if $W$ is spanned (over $F$) by elements of the form $y_{1}^{n_{1}}\cdots y_{l}^{n_{l}}$,
where $y_{i}\in Y$ and $l\leq h$. 
\end{defn}
The following fundamental theorem was proved by Shirshov. 
\begin{thm}
If an affine algebra $W$ has a multilinear PI of degree $t$, then
it has a Shirshov base of length $t$ and some height $h$ where $h$
depends only on $t$ and the number of generators of $W$. 
\end{thm}
In fact, there is an important special case where we can get even
``closer'' to representability.
\begin{thm}
\label{thm:finite-module-1} Let $C$ be a commutative algebra over
$F$ and let $W=C\left\langle a_{1},...,a_{s}\right\rangle $. Suppose
$W$ has a Shirshov base. If for every $i=1,\dots,s$, the element
$a_{i}$ is integral over $C$, then $W$ is a finite module over
$C$. 
\end{thm}
If in addition, our commutative algebra $C$ is Noetherian and unital
we reach our goal, as the next theorem shows.
\begin{thm}[Beidar \cite{Beidar1986}]
 Let $W$ be an algebra and $C$ be a unital commutative Noetherian
$F$-algebra. If $W$ is a finite module over $C$, then $W$ is representable. 
\end{thm}
Following the proof in \cite{AlexeiKanel-Belov} (Theorem 1.6.22),
it is easy to generalize this theorem to $H$-module algebras:
\begin{thm}
\label{thm:representable-over-commutattive-1} Let $W$ be an $H$-module
$C$-algebra, where $C$ is a unital commutative Noetherian $F$-algebra
(so in particular $h\cdot(cw)=c(h\cdot w)$ for $c\in C,h\in H$ and
$w\in W$). If $W$ is a finite module over $C$, then $W$ is $H$-representable
(i.e. there is a field extension $K$ of $F$ and an $H$-module $K$-algebra
$A$, which is finite dimensional over $K$, such that $W$ is $H$-module
$F$-subalgebra of $A$).\end{thm}
\selectlanguage{american}%
\begin{proof}
If $W$ does not posses an identity element we may replace $W$ by
$W\oplus C1$ ($H$ acts on $C1$ by $h\cdot c1=c\epsilon(h)1$).
Moreover, the map $\pi:C\to Z(W)$ given by $c\to c1$ is a homomorphism.
The image of $\pi$ is commutative unital Noetherian $F$-algebra.
Thus we may also assume that $C$ is embedded in the center of $W$. 

Next, if $I$ and $J$ are zero intersecting $H$-ideals of $W$,
we have $W\hookrightarrow W/I\times W/J$ is an $H$-module $C$-algebras
embedding. By Noetherian induction for $H$-ideals, we obtain that
$W$ is $H$-embedded in a finite product of Noetherian $H$-module
$C$-algebras each having no zero intersecting $H$-ideals ($H$-irreducible)
apart the zero ideals. Therefore, we may also assume $W$ is $H$-irreducible.

Suppose $z\in C$ is non nilpotent. Since $W$ is Noetherian there
is some $k$ for whcih $ann_{W}(z^{k})=ann_{W}(z^{k+1})=\cdots$.
Hence, $ann_{W}(z^{k})\cap z^{k}W=0$ (indeed, if $x=z^{k}w\in ann_{W}(z^{k})$,
then $z^{2k}w=0\Rightarrow w\in ann_{W}(z^{2k})=ann_{W}(z^{k})\Rightarrow x=0$).
Since $ann_{W}(z^{k})$ and $zW$ are $H$-ideals (recall that $h\cdot(zw)=z(h\cdot w)$)
and $W$ is an $H$-irreducible, we must conclude $ann_{W}(z^{k})=0$.
In other words, $z$ is not zero divisor in $W$.

Denote by $S$ all the non-nilpotent elements of $C$. By the previous
paragraph, $W$ $H$-embeds into $W_{1}=S^{-1}W$ (the $H$-action
is given by $h\cdot(s^{-1}w)=s^{-1}h(w)$). $C_{1}=S^{-1}C$ is Noetherian
and local (see Lemma 1.6.27 in \cite{AlexeiKanel-Belov}) with $J(C)$
equals to a nilpotent maximal ideal. Hence, by Lemma 16.25 in \cite{AlexeiKanel-Belov}
$C$ contains a field $K$ with the property $K\backsimeq C/J(C)$.

Denote by $k$ the nilpotency index of $J(C)$. So we have 
\[
J(C)\supseteq J(C)^{2}\supseteq\cdots\supseteq J(C)^{k-1}.
\]
Since $J(C)^{i}/J(C)^{i+1}$ is finite over $C$ (since $C$ is Noetherian),
it is also finite over $C/J(C)=K$. Hence $C$ is finite over $K$.
The theorem follows because $W_{1}$ is finite over $C$. 
\end{proof}
\selectlanguage{english}%

\section{Relatively free $H$-module algebra of a finite dimensional $H$-module
algebra\label{sub:Relatively-free--module}}

Suppose $A$ is a finite dimensional $H$-module $F$-algebra and
$\mathcal{A}$ is its corresponding affine relatively free $H$-module
algebra which is $H$-generated by the variables $x_{1},...,x_{t}$.
Suppose further that $a_{1},...,a_{l}$ is an $F$-basis for $A$
and consider the map $\phi:\mathcal{A}\to A\otimes_{F}K$, where $K=F\left(\left\{ \lambda_{i,j}|\, i=1,...,t;\,\, j=1...l\right\} \right)$,
induced by 
\[
x_{i}\to\sum_{k=1}^{l}\lambda_{i,j}a_{j}.
\]
It is easy to check that if we denote by $\mathcal{A}^{\prime}$ the
image of this $H$-map we will obtain that $\mathcal{A}$ and $\mathcal{A}^{\prime}$
are $H$-isomorphic. Therefore, we will abuse notation and denote
also the image by $\mathcal{A}$.

Suppose now that $A=A_{1}\times\cdots\times A_{s}$ is a product of
$H$-basic algebras. Denote by $R_{1},...,R_{s}$ the $H$-invariant
semisimple part of $A_{1},...,A_{s}$ respectively. We may embed $R_{i}$
into $End_{F}(R_{i})$ ($F$-algebras embedding) and define 
\[
tr:R_{1}\times\cdots\times R_{s}\to F^{\times s}
\]
by 
\[
tr(a_{1},...,a_{s})=(tr_{End_{F}(R_{!})}(a_{1}),...,tr_{End_{F}(R_{s})}(a_{s})).
\]
Furthermore, $tr$ can be extended to a function $A\to F^{\times s}$
by declaring that the trace of a radical element is $(0,...,0)$.
Since $F^{\times s}$ embeds into $R_{1}\times\cdots\times R_{s}$,
it acts on $A$. Finally, notice that each semisimple $a\in A$ satisfies
a Cayley-Hemilton identity of degree $d=\max\{d_{1}=\dim_{F}R_{1},...,d_{s}=\dim_{F}R_{s}\}$. 
\begin{lem}
\label{lem:traces_and_alternation}Suppose $f(x_{1},...,x_{d},Y)$
is an $H$-Kemer polynomial of rank at least $\mu_{A}+1$ and $\{x_{1},...,x_{d}\}$
is one of the small sets, then: 
\[
tr(a_{0})f(a_{1},...,a_{d},\bar{Y})=\sum_{k=1}^{d}f(a_{1},...,a_{k-1},a_{0}a_{k},a_{k+1},...,a_{d},\bar{Y}),
\]
where $a_{0},...,a_{d}\in A$ and $\bar{Y}$ is some evaluation of
the variables of $Y$ by elements of $A$.\end{lem}
\begin{proof}
See Proposition 10.5 in \cite{Aljadeff}.\end{proof}
\begin{cor}
\label{cor:Kemer_and_traces}If $I$ is an ideal of $\mathcal{A}$
generated (as an $H$-$T$-ideal) by $d$-alternating $H$-polynomials,
then $tr(x_{0})\cdot f\in I$, where $x_{0}\in\mathcal{A}$ and $f\in I$.
\end{cor}

\section{\label{sec:-Phoenix-property}$\Gamma$-Phoenix property}

Suppose $\Gamma$ is an $H$-$T$-ideal containing a Capelli identity.
We know this implies that $\Gamma$ contains the $H$-$T$-ideal of
a finite dimensional $H$-module $F$-algebra $A$. If we denote by
$p_{\Gamma}$ and $p_{A}$ the $H$-Kemer index of $\Gamma$ and $A$
respectively, then $p_{\Gamma}\leq p_{A}$. Our goal in this section
is to show that it is possible to replace $A$ by another finite dimensional
$H$-module algebra $B$ which is ``closer'' to $\Gamma$ in the
sense that its $H$-Kemer index and $H$-Kemer polynomials are exactly
as those of $\Gamma$. This will allow us to deduce the Phoenix property
for $H$-Kemer polynomials of $\Gamma$ from (the already established)
Phoenix property for $H$-Kemer polynomials of $B$.

Let $A$ be a finite dimensional $H$-module algebra which is a direct
product of basic algebras $A_{1}\times\cdots\times A_{s}$. Let $p_{A}$
and $p_{i}$ denote the $H$-Kemer index of $A$ and $A_{i}$, $i=1,\ldots,s$
respectively. We let $\mu_{i}=\mu_{A_{i}}$ and write $\mu_{0}$ for
the maximum of $\{\mu_{1},...,\mu_{s}\}$.
\begin{prop}
\label{prop:pre-phoenix} Let $\Gamma$ and $A$ as above. Then there
exist a representable $H$-module algebra $B$ with the following
properties: 
\begin{enumerate}
\item $id^{H}(B)\subseteq\Gamma$. 
\item The Kemer index $p_{B}$ of $B$ coincides with $p_{\Gamma}$. 
\item $\Gamma$ and $B$ have the same $H$-Kemer polynomials corresponding
to every $\mu$ which is $\geq\mu_{0}$. 
\end{enumerate}
Any $A$ satisfying $(2)$ and $(3)$ is called $H$-Kemer equivalent
to $W$. \end{prop}
\begin{cor}
By extending scalars to a larger field we may assume the $H$-module
algebra $B$ is finite dimensional over $F$.\end{cor}
\begin{proof}[of \propref{pre-phoenix}]
 Let $\mathcal{B}$ be a Shirshov base of $\mathcal{A}$. Consider
the constructions in \subref{Relatively-free--module} so that $\mathcal{A}$
is an $H$-module $F$-subalgebra of $A\otimes_{F}K$. Denote by $C$
the unital $F$-subalgebra of $K^{\times s}$ generated by the characteristic
values of the elements of $\mathcal{B}$. Notice that this is a Noetherian
$F$-algebra. Finally, define the $H$-module $C$-algebra $\mathcal{A}_{C}=C\cdot\mathcal{A}$.

Let $\mathcal{I}$ be the set of all evaluations in $\mathcal{A}$
of all $H$-Kemer polynomials of $A$ which are inside $\Gamma$.
It is clear $\mathcal{I}$ is an $H$-ideal of $\mathcal{A}$. By
\thmref{representable-over-commutattive-1} and \thmref{finite-module-1}
we know that $\mathcal{A}_{C}/C\mathcal{I}$ is representable. So,
since \corref{Kemer_and_traces} implies $\mathcal{A}/\mathcal{I}\subseteq\mathcal{A}_{C}/C\mathcal{I}$
we conclude that $\mathcal{A}/\mathcal{I}$ is representable.

Furthermore, $id^{H}(\mathcal{A}/\mathcal{I})\subseteq\Gamma$ and
$\mbox{Ind}(\mathcal{A}_{C}/\mathcal{CI})<\mbox{Ind}(A)$. So by extending
the field $F$ we are allowed to assume $\mathcal{A}/\mathcal{I}$
is a finite dimensional $H$-module algebra. By induction on the $H$-Kemer
index we obtain a finite dimensional (over some extension field of
$F$) $H$-module algebra which satisfies $(1)$ and $(2)$. In order
to get also $(3)$, we repeat the process above one final time.\end{proof}
\begin{cor}[Phoenix property]
\label{cor:phoneix_property} Let $A$ be an affine $H$-module algebra,
then there is a number $\mu_{W}^{\prime}$ such that every $H$-Kemer
polynomial $f$ of $A$ of degree at least $\mu_{W}^{\prime}$ satisfies
the $W$-Phoenix property.\end{cor}
\begin{proof}
By the previous theorem we may switch $W$ by a finite dimensional
$H$-module algebra without changing the $H$-Kemer index and polynomials.
So the corollary follows from \ref{Kemer_polynomials_of_A_are_Phoenix}.\end{proof}
\begin{defn}
Let $W$ be an affine $H$-module algebra and let $B$ be a finite
dimensional algebra as in \propref{pre-phoenix}. Denote by $\nu_{W}$
the number $\max\{\mu_{0},\mu_{W},\mu_{B},\mu_{W}^{\prime},\mu_{B}^{\prime}\}$.
Informally, for $\mu\geq\nu_{W}$ all the theorem concerning $H$-Kemer
polynomials of $W$ are true. 
\end{defn}

\section{Representable spaces\label{sec:Representable-Spaces}}

In this section we show the existence of a representable algebra $B_{\Gamma}$
satisfying the properties:
\begin{itemize}
\item $id^{H}(B_{\Gamma})\supseteq\Gamma$.
\item All $H$-Kemer polynomials of $\Gamma$ are non-identities of $B_{\Gamma}$.
\end{itemize}
\selectlanguage{american}%
We have seen in \secref{-Phoenix-property} that $W$ is $H$-Kemer
equivalent to a product of $H$-basic algebras $A=A_{1}\times\cdots\times A_{t}$.
Furthermore, \thmref{relativelly_affine} says that there is a number
$l$ such that the relatively free $H$-module algebra of $A$ on
the set $\Sigma=\{y_{1},...,y_{l}\}$ variables has the same $H$-identities
as $A$. Denote this algebra by $\mathcal{A}$. 

As before we identify $\mathcal{A}$ with an $H$-module subalgebra
of $A(\Lambda)$, where 
\[
\Lambda=\{\lambda_{k,i}|k=1...\mbox{dim}R,i=1...l\}.
\]

As in section \secref{-Phoenix-property} we view $R(\Lambda)$ as
a subalgebra of $End_{K}(R_{1\,}(\Lambda))\times\cdots\times End_{K}(R_{m}(\Lambda))$,
where $K=F(\Lambda),$ and consider the trace function $tr(a)=\left(tr(a),...,tr(a)\right)\in K^{\times m}$,
where $a$ is taken from $R(\Lambda)$. We may extend $tr$ to $A(\Lambda)$
by declaring the trace of a nilpotent element is zero.

Next, consider a Shirshov base $\mathcal{B}$ of $\mathcal{A}$ which
corresponds to the generators $x^{h}$ (here $x\in\Sigma$ and $h$
varies over a basis of $H$. This allows us to define the commutative
unital $F$-algebra $C$ generated by the characteristic values of
elements of $B$ (notice that each element in $A(\Lambda)$ satisfies
a Caylry-Hemiltonm polynomial of degree $d=\alpha(W)$). It is clear
that $C$ is Noetherian which acts on $R(\Lambda)$. 
\begin{defn}
Let $R$ be an $H$-module algebra. Denote by $id_{R}^{H}(A)$ the
$H$-ideal of $R$ consisting of all evaluations of polynomials in
$id^{H}(A)$ on $R$.

Denote by $(d,s)$ the $H$-Kemer index of $W$. Consider the relatively
free $H$-module algebra of $A$ on the set of variables $\Sigma\cup X$,
where $X=\cup_{i=1}^{\mu+1+r}X_{i}$, where $|X_{1}|=\cdots=|X_{\mu+1}|=d$
and $|X_{\mu+2}|=\cdots|X_{\mu+r+1}|=d+1$ and $\mu$ is big enough
so that $|X|\geq\dim J(A)$ and $\mu\geq\nu_{W}$. 

Recall the construction from (the proof of) \propref{property_K},
which is in our case:
\[
\mathcal{A}_{2}=\frac{A_{0}=C\cdot\mathcal{A}\ast_{C}C^{H}\left\{ X\right\} }{id_{A_{0}}^{H}(A)+\left\langle X\right\rangle _{H}^{|X|}}.
\]
We also define $\mathcal{A}_{1}$ to be the $H$-module algebra $H$-generated
by $\mathcal{A}$ and $X$. So: 
\[
\mathcal{A}_{1}=\frac{A_{0}^{\prime}=F^{H}\left\{ \Sigma\cup X\right\} }{id_{A_{0}^{\prime}}^{H}(A)+\left\langle X\right\rangle _{H}^{|X|}}.
\]

\end{defn}
It is easy to see that $id^{H}(\mathcal{A}_{1})=id^{H}(\mathcal{A})=id^{H}(A)$.
Note that this construction insures that any $H$-quotient of $\mathcal{A}_{2}$
is representable.
\selectlanguage{english}%
\begin{defn}
Let $f$ be an $H$-Kemer polynomial of $W$ with at least $\mu+1$
small sets. An evaluation of $f$ on $\mathcal{A}_{1}$ is admissible
if the following hold:
\begin{enumerate}
\item Precisely $\mu+1$ small sets in $f$, say $\dot{X}_{1},...,\dot{X}_{\mu+1}$,
are evaluated bijectively on the sets $X_{1},...,X_{\mu+1}$.
\item All big sets of $f$ are evaluated bijectively on the sets $X_{\mu+2},...,X_{\mu+1+r}$.
\item The rest of the variables in $f$ are evaluated on $H\cdot\mathcal{\Sigma}$. 
\end{enumerate}
\end{defn}
Denote by $\mathcal{S}$ the \textbf{set} of all admissible evaluations
of all $H$-Kemer polynomials of $W$.

Our goal is to prove that $\mathcal{S}$ projects injectively into
$\mathcal{A}_{3}=\mathcal{A}_{2}/id_{\mathcal{A}_{2}}^{H}(W)$. After
this is established it will be clear that:
\selectlanguage{american}%
\begin{enumerate}
\item $\mathcal{A}_{3}$ is representable.
\item $id^{H}(W)\subseteq id^{H}(\mathcal{A}_{3})$.
\item $W$ and $\mathcal{A}_{3}$ share the same $H$-Kemer polynomials.\end{enumerate}
\begin{lem}
$id_{\mathcal{A}_{2}}^{H}(W)\cap\mathcal{S}=\{0\}$.\end{lem}
\begin{proof}
Suppose $f\in\mathcal{S}$ is also in $id_{\mathcal{A}_{2}}^{H}(W)$.
So there are $g_{i}\in C$ and evaluations $p_{i}$ of multilinear
polynomials in $id^{H}(W)$ by elements from the set $X\cup\Sigma$
such that
\[
f=\sum g_{i}p_{i}.
\]
By specializing different $x\in X$ to $0$ we may assume that each
monomial of each $p_{i}$ has at least one appearance of every $x\in X$.
Since $\left\langle X\right\rangle _{H}^{|X|}=0$, each $p_{i}$ is
multilinear in $X$.

It is easy to check that $\mbox{Alt}{}_{X_{1}}(f)$ is well defined
on $\mathcal{A}_{1}$. So 
\[
d!\cdot f=\sum g_{i}\mbox{Alt}_{X_{1}}(p_{i}).
\]
Therefore, in $A_{0}/id_{A_{0}}^{H}(A)$ we get the equality: 
\[
d!\cdot f=\sum g_{i}\mbox{Alt}_{X_{1}}(p_{i})+b,
\]
where $b\in\left\langle X\right\rangle _{H}^{|X|}$. We may substitute
instead of each $x\in X$ an element of the form $\bar{x}=\sum_{k=1}^{\dim_{F}A}a_{k}\tau_{k,x}$,
where $\tau_{k,x}$ is a commutative indeterminate and $a_{1},...,a_{\dim A}$
is a basis of $A$. By \lemref{traces_and_alternation}, $g_{i}\overline{\mbox{Alt}_{X_{1}}(p_{i})}$
is equal to $\overline{\psi_{i}}$, where $\psi_{i}\in F^{H}\left\{ X\cup\Sigma\right\} $
(multilinear in $X$) is in $id^{H}(W)$. Thus, 
\[
d!\cdot f\equiv_{\mbox{mod}id^{H}(A)}\sum\psi_{i}\in id^{H}(W).
\]
Since $id^{H}(A)\subseteq id^{H}(W)$, we got a contradiction to $f$
being an $H$-Kemer polynomial of $W$.\end{proof}
\selectlanguage{english}%
\begin{cor}
\label{cor:Kemer polynomials_on_the_representable_algebra} Let $f$
be any $H$-Kemer polynomial of the $H$-module algebra $W$ (at least
$\mu+1$ small sets). Then $f\notin id^{H}(\mathcal{A}_{3})$.
\end{cor}

\section{Finalization of \thmref{affine-rep-1}}

We have all the ingredients needed to prove the main theorem.
\begin{proof}
The proof is by induction on the Kemer index $p$ associated to an
$H$-$T$-ideal $\Gamma$ (satisfying a Capelli identity). If $p=0$
then $\Gamma=F^{H}\left\langle X\right\rangle $ and so $W=0$. Suppose
the theorem is true for any affine $H$-module algebra with $H$-Kemer
index smaller than $p$. Denote by $S_{p}$ the $H$-$T$-ideal generated
by all $H$-Kemer polynomials corresponding to $\Gamma$, and let
$\Gamma'=\Gamma+S_{p}$. It is clear that the $H$-Kemer index of
$\Gamma'$ is smaller than $p$. Hence, by the inductive hypothesis
there is a representable $H$-module algebra $A'$ having $\Gamma'$
as its $H$-$T$-ideal of identities.

Let $B_{\Gamma}$ be the representable $H$-module algebra constructed
in the previous section. We'll show $\Gamma=id^{H}(A'\times B_{\Gamma})$.

It is clear that $\Gamma\subset id^{H}(A'\times B_{\Gamma})$ since
$\Gamma$ is contained in $\Gamma'$ and by construction $\Gamma\subseteq id^{H}(B_{\Gamma})$.
Suppose there is $f\notin\Gamma$ with $f\in id^{H}(A'\times B_{\Gamma})=id^{H}(A')\cap id^{H}(B_{\Gamma})$.
Since $f\in id^{H}(A')=\Gamma'$, we may assume $f\in S_{p}$. Using
the Phoenix property \corref{phoneix_property}, we obtain a Kemer
polynomial $f'$ (with at least $\mu+1$ small sets) such that $f'\in(f)_{H}$.
But by \corref{Kemer polynomials_on_the_representable_algebra}, $f\notin id^{H}(B_{\Gamma})$
and this contradicts our previous assumption on $f$. This completes
the proof.

\end{proof}

\subsection{Non affine case}

Let $H$ be any $F$-Hopf algebra. Denote by $H_{2}$ the Hopf algebra
$H\otimes_{F}\left(FC_{2}\right)^{*}$, where $C_{2}$ is the additive
group with two elements $0$ and $1$. An $F$-algebra $W$ is an
$H_{2}$-module algebra if it is $C_{2}$-graded $H$-module algebra
such that the graded component $W_{0}$ and $W_{1}$ of $W$ are stable
under the action of $H$. We denote by $G$ the Grassmann algebra
over $F$, which is a $C_{2}$-graded algebra. If $W$ is an $H_{2}$-module
algebra then we can define the $C_{2}$-graded algebra $E(W)=W_{0}\otimes E_{0}\oplus W_{1}\otimes E_{1}$.
It also has an $H$ structure given by 
\[
h\cdot\left(a_{0}\otimes w_{0}+a_{1}\otimes w_{1}\right)=h(a_{0})\otimes w_{0}+h(a_{1})\otimes w_{1}.
\]
Therefore we obtain an $H_{2}$-algebra. 

It is possible to get an $H_{2}$-module algebra from an $H$-module
algebra: Let $W$ be an $H$-module algebra. The algebra $W_{E}=W\otimes E$
is an $H_{2}$-module algebra, where $\left(W_{E}\right)_{0}=W\otimes E_{0}$
and $\left(W_{E}\right)_{1}=W\otimes E_{1}$. The $H$-action is given
by 
\[
h\cdot(w\otimes e)=(h\cdot w)\otimes e.
\]

The $H_{2}$-module algebra $F^{H_{2}}\left\{ X\right\} $ can be
considered as the $H$-module algebra $F^{H}\left\{ Y,Z\right\} $,
where $Y$ and $Z$ are countable sets of variables. the variables
in $Y$ are considered even and the ones in $Z$ are odd. We identify
$x_{i}\in X$ with $y_{i}+z_{i}$, thus for every $h\in H$ we have
$x_{i}^{h}=y_{i}^{h}+z_{i}^{h}$. Denote by $L_{d,l}$ the affine
$H_{2}$-module algebra $F^{H}{\left\{ y_{1},...,y_{d},z_{1},...,z_{l}\right\} }$

The following is proven in \cite{Berele1999}.
\begin{thm}
\label{thm:affine-rep}If $W$ is an $H$-module algebra which satisfy
an ordinary PI $f$, then $id^{H}(W)=id^{H}(E(L))$ for $L=L_{d,l}/id^{H_{2}}(W_{E})$,
where $d$ and $l$ are determined by the degree of $f$. 
\end{thm}
Since by \thmref{affine-rep-1} $L$ is $H_{2}$-PI equivalent to
a finite dimensional (over an extension field of $F$) $H_{2}$-module
algebra $A$, it is clear that:
\begin{thm}[\thmref{Hrep!!}]
\label{thm:rep}If $W$ is an $H$-module algebra which satisfy an
ordinary PI, then there is a finite dimensional $H_{2}$-module algebra
$A$ over some extension field of $F$, such that $id^{H_{2}}(W_{E})=id^{H_{2}}(A)$.
 
\end{thm}

\section{\label{sub:Specht-theorem-for}Specht theorem for $H$-module algebras}

In this section we prove \thmref{Specht!!!}:
\begin{thm}
\label{thm:Specht} Suppose $\Gamma$ is an $H$-$T$-ideal containing
an ordinary identity, then there are $f_{1},...,f_{s}\in\Gamma$ which
$H$-$T$-generate $\Gamma$. Equivalently, if $\Gamma_{1}\subseteq\Gamma_{2}\subseteq\cdots$
is an ascending chain of $H$-$T$-ideals containing an ordinary PI,
then the chain stabilizes. 
\end{thm}
Suppose $id^{H}(W_{i})=\Gamma_{i}$. By \thmref{affine-rep} $id^{H}(W_{i})=id^{H}(E(L_{d,l}/id^{H_{2}}(W_{i}\otimes E)))$,
where $d$ and $l$ are the same for all $W_{i}$. Moreover, it is
clear that $id^{H_{2}}(W_{1}\otimes E)\subseteq id^{H_{2}}(W_{2}\otimes E)\subseteq\cdots$
, so 
\[
id^{H_{2}}(E(L_{d,l}/id^{H_{2}}(W_{1}\otimes E)))\subseteq id^{H_{2}}(E(L_{d,l}/id^{H_{2}}(W_{2}\otimes E)))\subseteq\cdots.
\]
Therefore, it is enough to show:
\begin{thm}
\label{thm:specht-affine}If $\Gamma_{1}\subseteq\Gamma_{2}\subseteq\cdots$
is an ascending chain of $H$-$T$-ideals of an affine $H$-module
algebras containing an ordinary PI, then the chain stabilizes. 
\end{thm}
By \thmref{affine-rep-1} we can assume $id^{H}(A_{i})=\Gamma_{i}$,
where $A_{i}$ is a product of $H$-basic algebras. Moreover, since
the $H$-Kemer index is an order reversing function, the $H$-Kemer
index is eventually stabilizes. Thus, we suppose from the beginning
that all the $A_{i}$ have the same $H$-Kemer index $p=(d,s)$. Write
\[
A_{i}=\hat{A}_{i,1}\times\cdots\times\hat{A}_{i,r_{i}}\times\tilde{A}_{i},
\]
where $\bar{A}_{i,1},...,\bar{A}_{i,r_{i}}$ are $H$-basic of index
$p$ and $\tilde{A}_{i}$ is a product of $H$-basic algebras of lower
index. Since, due to \cite{Etingof2005}, the number of non $H$-isomorphic
$H$-semisimple algebras of dimension $d$ is finite, by passing to
a subsequence, we may also assume that there is a fixed set of $H$-semisimple
algebras $R_{1},...,R_{t}$ such that: 
\[
\{R_{1},...,R_{t}\}=\{\hat{A}_{i,1,ss},...,\hat{A}_{i,r_{i},ss}\}
\]
for all $i$ (here $\hat{A}_{i,r_{i},ss}$ is the semisimple part
of $\hat{A}_{i,r_{i}}$). 

Let 
\[
C_{j,i}=\frac{\bar{C_{j,i}}=R_{j}\ast F^{H}\left\{ X=\{x_{1},...,x_{s}\}\right\} }{id_{\bar{C}_{j,i}}^{H}(A_{i})+\left\langle X\right\rangle _{H}^{s}}
\]
for every $j=1...t$. Finally, write $C_{i}=C_{1,i}\times\cdots\times C_{t,i}$.
\begin{lem}
$id^{H}(C_{i}\times\tilde{A}_{i})=id^{H}(A_{i})$.\end{lem}
\begin{proof}
Clearly, $id^{H}(A_{i})\subseteq id^{H}(C_{i}\times\tilde{A}_{i})$.
Since $id^{H}(A_{i})=id^{H}(\hat{A}_{i,1}\times\cdots\times\hat{A}_{i,r_{i}}\times\tilde{A}_{i})$,
it is enough to show that $id^{H}(C_{i})\subseteq id^{H}(\hat{A}_{i,1}\times\cdots\times\hat{A}_{i,r_{i}})$.
So we prove that if $f=f(x_{1},...,x_{n})$ is a multilinear non-identity
of some $\hat{A}_{i,j}$, (say with semisimple part equal to $R_{1}$)
then $f$ is also a non-identity of $C_{1,i}$. Indeed, choose a non-zero
evaluation of $f$ by elements of $\hat{A}_{i,j}$. Suppose $\bar{x}_{1},..,\bar{x}_{l}$
are radical (so $l<s$) and the rest are semisimple (i.e. in $R_{1}$).
Thus, $f(x_{1},...,x_{l},\bar{x}_{l+1},...,\bar{x}_{n})$ is not zero
in $C_{1,i}$.\end{proof}
\begin{lem}
For $i$ large enough $C_{i}=C_{i+1}=\cdots(=B)$.\end{lem}
\begin{proof}
It is clear that $C_{i+1}$ is an $H$-epimorphic image of $C_{i}$,
so $\dim_{F}C_{i}\geq\dim C_{i+1}$. However, all the $C_{i}$ are
finite dimensional, so for $i$ large enough the sequence of dimensions
stabilizes. Hence, $C_{i+1}$ is $H$-isomorphic to $C_{i}$.
\end{proof}
We are ready to conclude the proof of \thmref{specht-affine} (and
thus also of \thmref{Specht}).

We are in the following situation:
\[
id^{H}(B\times\tilde{A}_{1})\subseteq id^{H}(B\times\tilde{A}_{2})\subseteq\cdots,
\]
where the index of $B$ is $p$ and the index of $\tilde{A}_{i}$
is smaller than $p$. Assume by induction on the $H$-Kemer index
that \thmref{specht-affine} holds for $H$-$T$-ideal of index smaller
than $p$. Denote by $I$ the $H$-$T$-ideal generated by all $H$-Kemer
polynomials of $B$. Clearly the sequence 
\[
id^{H}(B\times\tilde{A}_{1})+I\subseteq id^{H}(B\times\tilde{A}_{2})+I\subseteq\cdots
\]
stabilizes (by induction). Moreover, for all $i$ and $j$:
\[
id^{H}(B\times\tilde{A}_{i})\cap I=id^{H}(B\times\tilde{A}_{j})\cap I
\]
because $I\subseteq id^{H}(\tilde{A}_{i})\cap id^{H}(\tilde{A}_{j})$.
Therefore, the original sequence also stabilizes.

\section{\label{sub:The-exponent-of}The exponent of $H$-module algebras}

In this section we prove \thmref{exp}. 

Let $W$ be an $H$-module algebra satisfying an ordinary PI. By \thmref{Hrep!!}
we may assume that $W=G(A)$, where $A$ is an $H_{2}$-module finite
dimensional algebra. Gordienko in \cite{Gordienko2013} showed that
\thmref{exp} holds when $W$ is finite dimensional. So it will be
enough to prove that $\exp^{H}(W)=\exp^{H_{2}}(A)$. The idea is to
combine the following key theorem (of Gordienko) with the technique
of Giambruno and Zaicev in \cite{Giambruno2005} (Chapter 6.3).
\begin{thm}[Lemma 10 in \cite{Gordienko2013}]
 For almost all $n$ there is an $n$-multilinear $H_{2}$-polynomial
$f(X_{1},...,X_{\mu},X_{0})\notin id^{H_{2}}(A)$ such that $|X_{0}|<\alpha$
(for $\alpha$ not depended on $n$), $|X_{1}|=\cdots=|X_{\mu}|=\exp^{H_{2}}(A)$
and $f$ is alternating on each one of the $X_{i}$.
\end{thm}
The number $d$ in the above theorem can be computed in the following
fashion: Consider the $H$-Wedderburn-Malcev decomposition of $A$:
\[
A=J(A)\oplus R_{1}\times\cdots\times R_{q},
\]
where the $R_{i}$ are $H_{2}$-simple. Then 
\[
\exp^{H_{2}}(A)=\max\left\{ {\sum_{k=1}^{t}\dim R_{i_{k}}|R_{i_{1}}J(A)\cdots J(A)R_{i_{t}}\neq0\mbox{ and }i_{1},...,i_{t}\mbox{ are distinct}}\right\} .
\]
Moreover, we know that $A$ can be replaced by a product of $H$-basic
algebras $A_{1}\times\cdots\times A_{p}$. It is obvious that $\exp^{H_{2}}(A)=\max_{i}\exp^{H_{2}}(A_{i})$.
Furthermore, the exponent of $A_{i}$ (since $A_{i}$ is full) is
exactly $\alpha(A_{i})$ - the first component of the $H_{2}$-Kemer
index of $A_{i}$. Thus, $\exp^{H_{2}}(A)=\alpha(A)$. Using the construction
in \lemref{Kemer's_Lemma_2} and the previous theorem we obtain: 
\begin{thm}
For almost all $n$ there is an $n$-multilinear $H_{2}$-Kemer polynomial
$f=f(X_{1},...,X_{\mu+s},X_{0})\notin id^{H_{2}}(A)$ such that $|X_{0}\cup X_{\mu+1}\cup\cdots\cup X_{\mu+s}|<\beta$
(for $\beta$ not depended on $n$), $X_{1},...,X_{\mu}$ are the
small sets and $X_{\mu+1},...,X_{\mu+s}$ are the big sets.
\end{thm}
Suppose $A=A_{1}\times\cdots\times A_{r}$ is a product of $H_{2}$-simple
algebras and suppose $f$ from the previous theorem is an $H_{2}$-Kemer
polynomial of $A$. Therefore, $f$ is also an $H_{2}$-Kemer polynomial
of one of the $A_{i}$, say $A_{1}$. Denote by $B$ the $H_{2}$-algebra
$A_{ss}$. We may assume that each $X_{i}\in\{X_{0},...,X_{\mu+s}\}$
can be replaced by $Y_{i}\cup Z_{i}$, where $Y_{i}$ is a set of
even variables and $Z_{i}$ of odd variables, such that the resulting
polynomial is a non-identity of $A_{1}$. Surely, $|Y_{i}|=\dim B_{0}=d$
and $|Z_{i}|=\dim B_{1}=l$ for $i=1...\mu$. Indeed, consider a non-zero
substitution of $f$ by elements from $B_{0}\cup B_{1}\cup J(A_{1})$.
Each big set of $f$ has $\dim B+1$ elements, so any non-zero substitution
must include a radical element. There are $s$ big sets and the nilpotency
index of $J(A_{1})$ is $s+1$, hence all the other substitutions
of variables of $f$ must be semisimple. Moreover, since $f$ is alternating
on each one of the small sets, we must substitute a full basis in
each of them. Thus, $d$ variables of each $X_{i}$ must be even and
the rest odd. Let us call the new polynomial also $f$.
\begin{defn}
Suppose $f=f(y_{1},...,y_{n},z_{1},...,z_{s})$ is a multilinear polynomial
in $F^{H}\left\{ Y,Z\right\} $. We may write 
\[
f=\sum a_{\sigma,W,\mathbf{h}}W_{0}z_{\sigma(1)}^{h_{1}}W_{1}z_{\sigma(2)}^{h_{2}}\cdots z_{\sigma(s)}^{h_{s}}W_{s},
\]
where the sum runs on $\sigma\in S_{s}$, $W=(W_{0},...,W_{s})$ -
$s$-tuple of monomials (we allow the zero monomial: $1$) in distinct
variables from $\{y_{1},...,y_{n}\}$ and $\mathbf{h}=(h_{1},...,h_{s})\in\{b_{1},...,b_{m}\}^{s}$
(recall that $b_{1},...,b_{m}$ is a basis of $H$). Then $\tilde{f}$
is defined to be the polynomial
\[
\sum(-1)^{\sigma}a_{\sigma,W,\mathbf{h}}W_{0}z_{\sigma(1)}^{h_{1}}W_{1}z_{\sigma(2)}^{h_{2}}\cdots z_{\sigma(s)}^{h_{s}}W_{s}.
\]

\end{defn}
The following is folklore:
\begin{lem}
Suppose $f$ is a multilinear polynomial in $F^{H}\left\{ Y,Z\right\} $.
Then the following holds:
\begin{enumerate}
\item $f\in id^{H_{2}}(W)\Longleftrightarrow\tilde{f}\in id^{H_{2}}(G(A)).$
\item $c_{n}^{H_{2}}(W)=c_{n}^{H_{2}}(G(W)).$
\end{enumerate}
\end{lem}
The next key Proposition relies on the representation theory of $S_{n}$.
The reader is advised to review \subref{Affine-relatively--module}
for a notation reminder.

\global\long\def\l{\lambda^{\prime}}
 \global\long\def\ll{\lambda^{\prime\prime}}

\begin{prop}
Let $f$ be the polynomial from the previous theorem and let $g=\tilde{f}$.
Then for $\lambda^{\prime}=(\mu^{d})$ and $\lambda^{\prime\prime}=(l^{\mu})$,
there are $T_{\lambda^{\prime}}$ and $T_{\lambda^{\prime\prime}}$
such that $e_{T_{\l}}e_{T_{\ll}}g\notin id^{H_{2}}(E(A))$.\end{prop}
\begin{proof}
Since $g\notin id^{H_{2}}(E(A))$, there is some $\l=(\lambda_{1}^{\prime},...)\vdash d\mu$
and $T_{\l}$ indexed by variables from $Y=Y_{1}\cup\cdots\cup Y_{\mu}$
such that $e_{T_{\l}}g\notin id^{H_{2}}(E(A))$. We claim that $\l$
is of the shape $(\mu^{d})$:

Write 
\[
e_{1}=\sum_{\sigma\in\mathcal{R}_{T_{\l}}}\sigma,\,\, e_{2}=\sum_{\sigma\in\mathcal{C}_{T_{\l}}}(-1)^{\sigma}\sigma,
\]
 so $e_{T_{\l}}=e_{1}e_{2}$. 

If $\lambda_{1}^{\prime}>d$ , then $e_{T_{\l}}g$ is symmetric on
(at least) $\mu+1$ variables from $Y$. Thus, for every $\sigma\in\mathcal{C}_{T_{\l}}$
at least two of them must fall in the same $\sigma(Y_{i})$. However,
$\sigma g$ is alternating on $\sigma(Y_{i})$, so $e_{1}\sigma g=0\Rightarrow e_{T_{\l}}g=e_{1}e_{2}g=0$
- contradiction.

Suppose $h(\l)>d$, where $h\{\l)$ is the height of $T_{\l}$. Therefore,
$e_{T_{\l}}g$ is alternating on a $d+1$ subset $Y^{\prime}$ of
$Y$. Thus $\widetilde{e_{T_{\l}}g}=e_{T_{\l}}\tilde{g}=e_{T_{\l}}f$
is also alternating on $Y^{\prime}$. On any non-zero substitution
of $f$ (on every $A_{i}$) the radical values appear only in $X_{\mu+1},...,X_{\mu+t}$.
Hence, the non-zero evaluations of $Y^{\prime}$ must consist of semisimple
(even) elements. However, the dimension of the even semisimple part
is $d$, thus $e_{2}e_{1}f\in id^{H_{2}}(A)$, hence $e_{T_{\l}}g\in id^{H_{2}}(E(A))$
- contradiction. All in all, $\l$ must be equal to $(\mu^{d})$.

There is some $\ll=(\lambda_{1}^{\prime\prime},...)\vdash l\mu$ and
$T_{\ll}$ indexed by variables from $Z=Z_{1}\cup\cdots\cup Z_{\mu}$
such that $e_{T_{\ll}}g\notin id^{H_{2}}(G(A))$. As before, we plan
to demonstrate that $\ll$ has the shape $(\mu^{d})$.

Suppose $h(\ll)>\mu$. Write $e_{1}=\sum_{\sigma\in\mathcal{R}_{T_{\ll}}}\sigma$
and $e_{2}=\sum_{\sigma\in\mathcal{C}_{T_{\ll}}}(-1)^{\sigma}\sigma$
so that $e_{T_{\ll}}=e_{1}e_{2}$. $e_{2}g$ is alternating on (at
least) $\mu+1$ variables from $Z$. Thus, at least two of them must
fall in some $Z_{i}$. Since $e_{2}g$ is also symmetric on each $Z_{i}$,
$e_{2}g=0$. So $e_{T_{\ll}}g=0$.

Suppose $\lambda_{1}^{\prime\prime}>l$. Then, $e_{T_{\ll}}g$ is
symmetric on $l+1$ elements of $Z$, say $Z^{\prime}$. Hence $\widetilde{e_{T_{\ll}}g}=\tilde{e}_{T_{\ll}}\tilde{g}$
is alternating on $Y^{\prime}$. where 
\[
\tilde{e}_{T_{\ll}}=\sum_{\tau\in\mathcal{R}_{T_{\ll}}}\sum_{\sigma\in\mathcal{C}_{T_{\ll}}}(-1)^{\sigma}\sigma\tau=\sum_{\tau\in\mathcal{R}_{T_{\ll}}}\sum_{\sigma\in\mathcal{C}_{T_{\ll}}}(-1)^{\sigma}\tau\sigma.
\]
As before, this implies that $\widetilde{e_{T_{\ll}}g}\in id^{H_{2}}(A)$,
hence $e_{T_{\ll}}g\in id^{H_{2}}(E(A))$ - contradiction. All in
all, $\ll$ must be equal to $(l^{\mu})$.

Because, $e_{T_{\l}}$and $e_{T_{\ll}}$act on distinct sets of variables,
we conclude
\[
e_{T_{\l}}e_{T_{\ll}}g\notin id^{H}(E(A)).
\]
 .\end{proof}
\begin{cor}
If we replace the graded variables of $g$ by non graded ones, we
obtain an $H$-polynomial (which we continue to denote by $g$) which
is not in $id^{H}(E(A))$. Moreover, $e_{T_{\l}}e_{T_{\ll}}g\notin id^{H}(E(A))$.\end{cor}
\begin{defn}
Let $t>0$ be an integer number. \nomenclature[999]{$h(d,l,t)$}{Tthe partition of $d(l+t)+tl$  given by: $(\underset{d\mbox{ times}}{\underbrace{l+t,...,l+t}},\underset{t\mbox{ times}}{\underbrace{l,...,l}})$\\ }$h(d,l,t)$
is the partition of $d(l+t)+tl$ given by: 
\[
(\underset{d\mbox{ times}}{\underbrace{l+t,...,l+t}},\underset{t\mbox{ times}}{\underbrace{l,...,l}}).
\]
The corresponding tableau looks like: 

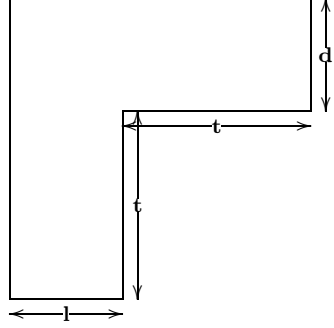
\begin{xy}<1cm,0cm>:
(22,0)="A";
(22,4)="B" **@{-};
(26,4)="C" **@{-};
(26,2.5)="D" **@{-};
(23.5,2.5)="E" **@{-};
(23.5,0)="F" **@{-}; 
"A" **@{-},
\ar@{<->}|-{\textbf{d}} "C" + (0.2,0);  "D" + (0.2,0),
\ar@{<->}|-{\textbf{t}} "E" + (0.2,0);  "F" + (0.2,0),
\ar@{<->}|-{\textbf{t}} "E" - (0,0.2);  "D" - (0,0.2),
\ar@{<->}|-{\textbf{l}} "A" - (0,0.2);  "F" - (0,0.2),
\end{xy}
\end{defn}
\begin{cor}
For any $\mu>0$, there exist a partition $\lambda$ of $\mu(d+l)$
such that 
\[
h(d,l,\mu-d-l)\leq\lambda\leq h(d,l,\mu)
\]
and for $n=\mu(d+l)+\beta$ the space $P_{n}^{H}/P_{n}^{H}\cap id^{H}(G(A))$
contains an irreducible $S_{|\lambda|}$-module corresponding to $\lambda$.
Furthermore, 
\[
c_{n}^{H}(G(A))\geq Cn^{\gamma}\left(\exp^{H_{2}}(A)\right)^{n},
\]
for some numbers $C$ and $\gamma$. \end{cor}
\begin{proof}
The first part follows from the previous corollary by means of the
Littlewood-Richardson rule (see Theorem 2.3.9 in \cite{Giambruno2005}).
The second part follows from the first part and Lemma 6.2.5 in \cite{Giambruno2005}.
\end{proof}
To finish the proof of \thmref{exp}, it is suffice to establish that
$c_{n}^{H}(E(A))\leq c_{n}^{H_{2}}(A)$. Indeed, for every $H_{2}$-module
algebra $A$, 
\[
\frac{P_{n}^{H}}{P_{n}^{H}\cap id^{H}(A)}\hookrightarrow\frac{P_{n}^{H_{2}}}{P_{n}^{H_{2}}\cap id^{H_{2}}(A)},
\]
where the map is induced by $x_{i}\to y_{i}+z_{i}$. 

Thus, $c_{n}^{H}(E(A))\leq c_{n}^{H_{2}}(E(A))$. Since $c_{n}^{H_{2}}(E(A))=c_{n}^{H_{2}}(A)$,
we obtain $c_{n}^{H}(E(A))\leq c_{n}^{H_{2}}(A)$.

\title{\bibliographystyle{plain}
\phantomsection\addcontentsline{toc}{section}{\refname}\bibliography{Ref}
}
\end{document}